\documentclass[12pt]{amsart}
\usepackage[margin=2cm]{geometry}
\usepackage{latexsym}
\usepackage{amsmath}
\usepackage{amsfonts}
\usepackage{amssymb}
\usepackage[all]{xy}
\usepackage{pgf,tikz}
\usetikzlibrary{arrows}

\newtheorem{theorem}{Theorem}[section]
\newtheorem{lemma}[theorem]{Lemma}

\theoremstyle{definition}
\newtheorem{definition}[theorem]{Definition}

\theoremstyle{remark}

\numberwithin{equation}{section}

\begin{document}
\definecolor{zzttqq}{rgb}{0.6,0.2,0}
\definecolor{qqqqff}{rgb}{0,0,1}
\definecolor{darkspringgreen}{rgb}{0.09, 0.45, 0.27}
\title{Maximal Green Sequences for Cluster Algebras Associated to the $n$-Torus}

\author{Eric Bucher}
\address{Mathematics Department\\
Louisiana State University\\
Baton Rouge, Louisiana}
\email{ebuche2@tigers.lsu.edu}

\date{December 11th, 2014}

\begin{abstract}
Given a marked surface $(S,M)$ we can add arcs to the surface to create a triangulation, $T$, of that surface. For each triangulation, $T$, we can associate a cluster algebra. In this paper we will consider the torus of genus $n$ with two interior marked points (called punctures). We will construct a specific triangulation of this surface which yeilds a specific quiver. Then in the sense of work by Keller we will produce a \textit{maximal green sequence} for this quiver.

\end{abstract}

\maketitle

\section{Introduction}
Cluster algebras were invented by Fomin and Zelevinsky \cite{cluster1} in 2003. Within a very short period of time cluster algebras became an important tool in the study of phenomena in various areas of mathematics and mathematical physics. They play an important role in the study of Teichm\"{u}ller theory, canonical bases, total positivity, Poisson Lie-groups, Calabi-Yau algebras, noncommutative Donaldson-Thomas invariants, scattering amplitudes,  and representations of finite dimensional algebras. For more information on the diverse scope of cluster algebras see the review paper by Williams \cite{williams}. 
\\

The idea of maximal green sequences of cluster mutations was introduced by Keller in \cite{keller}. He explored important applications of this notion, by utilizing it in the explicit computation of noncommutative Donaldson-Thomas invariants of triangulated categories which were introduced by Kontsevich and Soibelman in \cite{kontsevich}. Addtionally, Alim, et al worked with this notion in connection with the computation of spectra of BPS states \cite{alim}. Very recently this notion also played a key role in the Gross-Hacking-Keel-Kontsevich \cite{gross} proof of the full Fock-Goncharov conjecture for large classes of cluster algebras.
\\

The problem of existence of maximal green sequences of cluster mutations is difficult due to the interative nature of the choices of mutations. The iterative nature of the problem means that exhaustive methods are not always effective when searching for a maximal green sequence. In spite of this difficulty there has been a vast amount of progress made in the area. Br\"{u}stle, Dupont, and Perotin proved the existence of maximal green sequences for cluster algebras of finite type in \cite{brustle}. Alim et al. showed that cluster algebras from surfaces with nonempty boundry have a maximal green sequence \cite{alim}. Yakimov proved the existence of maximal green sequences for the Berenstein-Fomin-Zelevinsky cluster algebras on all double Bruhat cells in Kac-Moody groups in \cite{yakimov}. Also, Garver and Musiker constructed maximal green sequences for all type A quivers in \cite{garver}. 
\\

In general a cluster algebra can be constructed from any orientable surface by looking at the possible triangulations of that surface. This construction is introduced by Gekhtman, Shapiro, and Vainshtein in \cite{gekhtman} and in a more general setting by Fock and Goncharov in \cite{fock}. This construction is extremely important because any cluster algebra of finite mutation type can be realized as a cluster algebra which arises from a surface following this construction. An important problem in cluster algebras is then to prove the existence or non-existence of maximal green sequences for each cluster algebra which arises from the triangulation of a surface. This paper will prove the existence of maximal green sequences for an infinite family of cluster algebras which arise this way. For a more in depth look into the procedure of creating a cluster algebra from a triangulated surface see the work by Fomin, Shapiro, and Thurston \cite{fomin}. 
\\

In this paper we prove the existence of a maximal green sequence for cluster algebras which arise from triangulations of the twice punctured $n$-torus. This is an infinite family of cluster algebras for which we explicitly find the maximal green sequence. We will start with an $n$-torus. We then construct a specific triangulation of this surface, $T_n$. This triangulation is chosen to contain a large amount of symmetry, which will play an integral part in our main proof. The construction of this triangulation will be discussed in section 3. After constructing the triangulation, we look at the quiver it correlates to, $Q_{T_n}$. We take advantage of the symmetry of this quiver, by breaking it into smaller parts. This cluster algebra contains a large $n$ cycle with identical subquivers attached to each vertex. We construct a green sequence for the cycle, which leaves the attached subquivers unaffected. We can then apply a green sequence to the subquivers which will minimally effect the vertices on the cycle. Various mutations are then done to correct these minimal effects. We want to emphasize that the ability to correct these effects is directly related to the choice of triangulation. By creating subquivers of a certain structure we can gurantee that they will not be drasticatly affected by the sequence of mutations applied to the interconnecting cycle. The combining of these sequences will result in a maximal green sequence for the quiver $Q_{T_n}$.  In essence, we are creating seperate maximal green sequences for each "piece" of the quiver and then creating a procedure for gluing these sequences together. The details of the proof are presented in section 4 of this paper. Before beginning we need to establish some background definitions and notation.

\section{Preliminaries}

We will follow the notation laid out by Br\"{u}stle, Dupont, and Perotin \cite{brustle}.
\begin{definition} A \textbf{quiver}, $Q$, is a directed graph containing no $2$-cycles or loops.
\end{definition}

The notation $Q_0$ will denote the vertices of $Q$. Also, $Q_1$ will denote the edges of $Q$ which are referred to as \textbf{arrows}. We will let $Q_0=[N]$.

\begin{definition} An \textbf{ice quiver} is a pair $(Q,F)$ where Q is a quiver as described above and $F \subset Q_0$ is a subset of vertices called frozen vertices; such that there are no arrows between them. For simplicity, we always assume that $Q_0=\{1,2,3, \dots ,n+m\}$ and that $F=\{n+1,n+2,\dots, n+m\}$ for some integers $n,m \geq 0$. If $F$ is empty we write $(Q,\emptyset)$ for the ice quiver. 
\end{definition}

In this paper we will be concerned with a process called mutation. Mutation is a process of obtaining a new ice quiver from an existing one. 

\begin{definition} Let $(Q,F)$ be an ice quiver and $k\in Q_0$ a non-frozen vertex. The \textbf{mutation} of a quiver $(Q,F)$ at a vertex $k$ is denoted $\mu_k$, and produces a new ice quiver $(\mu_k(Q),F)$. The vertices of $(\mu_k(Q),F)$ are the same vertices from $(Q,F)$. The arrows of the new quiver are obtained by performing the following $3$ steps:
\begin{enumerate}
\item For every 2-path $i \rightarrow k \rightarrow j$ , adjoin a new arrow $i \rightarrow j$.
\item Reverse the direction of all arrows incident to $k$.
\item Delete any 2-cycles created during the first two steps as well as any arrows created between frozen vertices.
\end{enumerate}
\end{definition}

It is important to note that we do not allow mutation at a frozen vertex. We will denote $Mut(Q)$ to be the set of all quivers who can be obtained from $Q$ by a sequence of mutations.

The ice quivers which are of concern in this paper have a very specific set of frozen vertices. We will be looking at what are referred to as the framed and coframed quivers associated to $Q$.

\begin{definition} 
The \textbf{framed quiver} associated with $Q$ is the quiver $\hat{Q}$ such that: 

$$\hat{Q}_0 = Q_0 \sqcup \{i'\text{ }|\text{ }i\in Q_0\}$$
$$\hat{Q}_1 = Q_1 \sqcup \{i \to i'\text{ }|\text{ }i \in Q_0\}$$

\flushleft The \textbf{coframed quiver} associated with $Q$ is the quiver $\breve{Q}$ such that:
$$\breve{Q}_0 = Q_0 \sqcup \{i'\text{ }|\text{ }i\in Q_0\}$$
$$\breve{Q}_1 = Q_1 \sqcup \{i' \to i\text{ }|\text{ }i \in Q_0\}$$
\end{definition}
Both quivers $\hat{Q}$ and $\breve{Q}$ are naturally ice quivers whose frozen vertices are commonly written as $\hat{Q}_0'$ and $\breve{Q}_0'$. Next we will talk about what it means for a vertex to be green or red.

\begin{definition}
Let $R \in Mut(\hat{Q})$. A non-frozen vertex $i \in R_0$ is called \textbf{green} if $$\{j'\in Q_0'\text{ }| \text{ } \exists \text{ } j' \rightarrow i \in R_1 \}=\emptyset.$$ It is called \textbf{red} if $$\{j'\in Q_0'\text{ }| \text{ } \exists \text{ } j' \leftarrow i \in R_1 \}=\emptyset.$$ 
\end{definition}

In \cite{brustle} they show that every non-frozen vertex in $R_0$ is either red or green. This idea is what motivates our work in this paper. It arises as a question of green sequences. 

\begin{definition}
A \textbf{green sequence} for $Q$ is a sequence $\textbf{i}=\{i_1, \dots, i_l\} \subset Q_0$ such that $i_1$ is green in $\hat{Q}$ and for any $2\leq k \leq l$, the vertex $i_k$ is green in $\mu_{i_{k-1}}\circ \cdots \circ \mu_{i_1}(\hat{Q})$. The integer $l$ is called the length of the sequence $\textbf{i}$ and is denoted by $l(\textbf{i})$.\\
\\
A green sequence \textbf{i} is called maximal if every non-frozen vertex in $\mu_{\textbf{i}}(\hat{Q})$ is red where $\mu_{\textbf{i}}=\mu_{i_{l}}\circ \cdots \circ \mu_{i_1}$. We denote the set of all maximal green sequences for $Q$ by $$\text{green}(Q)=\{\textbf{i } | \textbf{ i}\text{ is a maximal green sequence for }Q\}.$$  
\end{definition}

In this paper we will construct a maximal green sequence for a specific infinite family of quivers which will be described in the following section. In essence what we want to show is that green$(Q)\neq \emptyset$ for each quiver, $Q$, in this family. In order to do this we must first discuss where our quivers are coming from.

\section{Constructing the Triangulation $T_n$}
In work by Fomin, Shapiro, and Thurston \cite{fomin} there is a very precise description of how you can associate a quiver $Q$ to a triangulated surface. The surfaces that we will be discussing in this paper are twice punctured tori. We will find a specific triangulation on those surfaces which we will denote $T_n$. By following the techniques outlined in \cite{fomin} from there we will form the associated quiver which we will denote by $Q_{T_n}$. 
\\
\\Start by letting $(S,M)=$ the torus of genus $n$ with two interior marked points. Now we will construct the desired traingulation $T_n$ for the marked surface $(S,M)$.  We start by drawing $(S,M)$ as the identification space below.

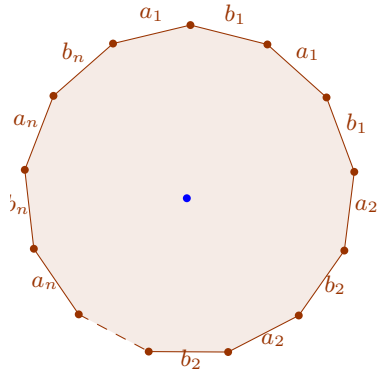
\[\begin{tikzpicture}[line cap=round,line join=round,>=triangle 45,x=0.3cm,y=0.3cm]
\clip(-11.54,-21.4) rectangle (5.26,-4.41);
\fill[color=zzttqq,fill=zzttqq,fill opacity=0.1] (-5.42,-20.19) -- (-1.9,-20.21) -- (1.23,-18.59) -- (3.25,-15.71) -- (3.69,-12.22) -- (2.46,-8.92) -- (-0.16,-6.57) -- (-3.57,-5.71) -- (-7,-6.53) -- (-9.64,-8.85) -- (-10.91,-12.13) -- (-10.51,-15.63) -- (-8.52,-18.54) -- cycle;
\draw [color=zzttqq] (-5.42,-20.19)-- (-1.9,-20.21);
\draw [color=zzttqq] (-1.9,-20.21)-- (1.23,-18.59);
\draw [color=zzttqq] (1.23,-18.59)-- (3.25,-15.71);
\draw [color=zzttqq] (3.25,-15.71)-- (3.69,-12.22);
\draw [color=zzttqq] (3.69,-12.22)-- (2.46,-8.92);
\draw [color=zzttqq] (2.46,-8.92)-- (-0.16,-6.57);
\draw [color=zzttqq] (-0.16,-6.57)-- (-3.57,-5.71);
\draw [color=zzttqq] (-3.57,-5.71)-- (-7,-6.53);
\draw [color=zzttqq] (-7,-6.53)-- (-9.64,-8.85);
\draw [color=zzttqq] (-9.64,-8.85)-- (-10.91,-12.13);
\draw [color=zzttqq] (-10.91,-12.13)-- (-10.51,-15.63);
\draw [color=zzttqq] (-10.51,-15.63)-- (-8.52,-18.54);
\draw [dash pattern=on 4pt off 4pt,color=zzttqq] (-8.52,-18.54)-- (-5.42,-20.19);
\begin{scriptsize}
\fill [color=zzttqq] (-5.42,-20.19) circle (1.5pt);
\fill [color=zzttqq] (-1.9,-20.21) circle (1.5pt);
\draw[color=zzttqq] (-3.53,-20.54) node {$b_2$};
\draw[color=zzttqq] (0.1,-19.68) node {$a_2$};
\draw[color=zzttqq] (2.84,-17.22) node {$b_2$};
\draw[color=zzttqq] (4.24,-13.74) node {$a_2$};
\draw[color=zzttqq] (3.81,-10.07) node {$b_1$};
\draw[color=zzttqq] (1.63,-6.99) node {$a_1$};
\draw[color=zzttqq] (-1.58,-5.23) node {$b_1$};
\draw[color=zzttqq] (-5.29,-5.23) node {$a_1$};
\draw[color=zzttqq] (-8.73,-6.91) node {$b_n$};
\draw[color=zzttqq] (-10.84,-9.96) node {$a_n$};
\draw[color=zzttqq] (-11.23,-13.67) node {$b_n$};
\draw[color=zzttqq] (-10.02,-17.14) node {$a_n$};
\draw[color=zzttqq] (-7.09,-19.64) node {};
\fill [color=zzttqq] (1.23,-18.59) circle (1.5pt);
\fill [color=zzttqq] (3.25,-15.71) circle (1.5pt);
\fill [color=zzttqq] (3.69,-12.22) circle (1.5pt);
\fill [color=zzttqq] (2.46,-8.92) circle (1.5pt);
\fill [color=zzttqq] (-0.16,-6.57) circle (1.5pt);
\fill [color=zzttqq] (-3.57,-5.71) circle (1.5pt);
\fill [color=zzttqq] (-7,-6.53) circle (1.5pt);
\fill [color=zzttqq] (-9.64,-8.85) circle (1.5pt);
\fill [color=zzttqq] (-10.91,-12.13) circle (1.5pt);
\fill [color=zzttqq] (-10.51,-15.63) circle (1.5pt);
\fill [color=zzttqq] (-8.52,-18.54) circle (1.5pt);
\fill [color=qqqqff] (-3.73,-13.39) circle (1.5pt);
\end{scriptsize}
\end{tikzpicture}\]

After we have created the identification space we want to add additional arcs to create a triangulation of this space. At the moment the set of arcs we will be using in our triangulation are $a_1,b_1, a_2,b_2\dots, a_n,b_n$. The additional arcs we wish to add can be seen in the diagram below.

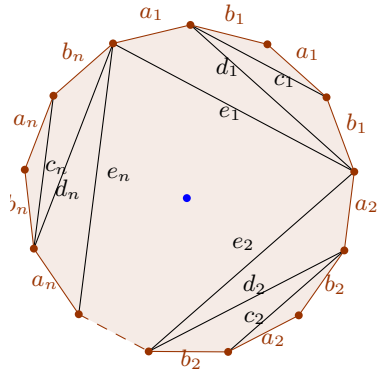
\[\begin{tikzpicture}[line cap=round,line join=round,>=triangle 45,x=0.3cm,y=0.3cm]
\clip(-11.54,-21.4) rectangle (5.26,-4.41);
\fill[color=zzttqq,fill=zzttqq,fill opacity=0.1] (-5.42,-20.19) -- (-1.9,-20.21) -- (1.23,-18.59) -- (3.25,-15.71) -- (3.69,-12.22) -- (2.46,-8.92) -- (-0.16,-6.57) -- (-3.57,-5.71) -- (-7,-6.53) -- (-9.64,-8.85) -- (-10.91,-12.13) -- (-10.51,-15.63) -- (-8.52,-18.54) -- cycle;
\draw [color=zzttqq] (-5.42,-20.19)-- (-1.9,-20.21);
\draw [color=zzttqq] (-1.9,-20.21)-- (1.23,-18.59);
\draw [color=zzttqq] (1.23,-18.59)-- (3.25,-15.71);
\draw [color=zzttqq] (3.25,-15.71)-- (3.69,-12.22);
\draw [color=zzttqq] (3.69,-12.22)-- (2.46,-8.92);
\draw [color=zzttqq] (2.46,-8.92)-- (-0.16,-6.57);
\draw [color=zzttqq] (-0.16,-6.57)-- (-3.57,-5.71);
\draw [color=zzttqq] (-3.57,-5.71)-- (-7,-6.53);
\draw [color=zzttqq] (-7,-6.53)-- (-9.64,-8.85);
\draw [color=zzttqq] (-9.64,-8.85)-- (-10.91,-12.13);
\draw [color=zzttqq] (-10.91,-12.13)-- (-10.51,-15.63);
\draw [color=zzttqq] (-10.51,-15.63)-- (-8.52,-18.54);
\draw [dash pattern=on 4pt off 4pt,color=zzttqq] (-8.52,-18.54)-- (-5.42,-20.19);
\draw (-7,-6.53)-- (3.69,-12.22);
\draw (3.69,-12.22)-- (-3.57,-5.71);
\draw (-3.57,-5.71)-- (2.46,-8.92);
\draw (3.69,-12.22)-- (-5.42,-20.19);
\draw (-5.42,-20.19)-- (3.25,-15.71);
\draw (3.25,-15.71)-- (-1.9,-20.21);
\draw (-8.52,-18.54)-- (-7,-6.53);
\draw (-7,-6.53)-- (-10.51,-15.63);
\draw (-10.51,-15.63)-- (-9.64,-8.85);
\begin{scriptsize}
\fill [color=zzttqq] (-5.42,-20.19) circle (1.5pt);
\fill [color=zzttqq] (-1.9,-20.21) circle (1.5pt);
\draw[color=zzttqq] (-3.53,-20.54) node {$b_2$};
\draw[color=zzttqq] (0.1,-19.68) node {$a_2$};
\draw[color=zzttqq] (2.84,-17.22) node {$b_2$};
\draw[color=zzttqq] (4.24,-13.74) node {$a_2$};
\draw[color=zzttqq] (3.81,-10.07) node {$b_1$};
\draw[color=zzttqq] (1.63,-6.99) node {$a_1$};
\draw[color=zzttqq] (-1.58,-5.23) node {$b_1$};
\draw[color=zzttqq] (-5.29,-5.23) node {$a_1$};
\draw[color=zzttqq] (-8.73,-6.91) node {$b_n$};
\draw[color=zzttqq] (-10.84,-9.96) node {$a_n$};
\draw[color=zzttqq] (-11.23,-13.67) node {$b_n$};
\draw[color=zzttqq] (-10.02,-17.14) node {$a_n$};
\draw[color=zzttqq] (-7.09,-19.64) node {};
\fill [color=zzttqq] (1.23,-18.59) circle (1.5pt);
\fill [color=zzttqq] (3.25,-15.71) circle (1.5pt);
\fill [color=zzttqq] (3.69,-12.22) circle (1.5pt);
\fill [color=zzttqq] (2.46,-8.92) circle (1.5pt);
\fill [color=zzttqq] (-0.16,-6.57) circle (1.5pt);
\fill [color=zzttqq] (-3.57,-5.71) circle (1.5pt);
\fill [color=zzttqq] (-7,-6.53) circle (1.5pt);
\fill [color=zzttqq] (-9.64,-8.85) circle (1.5pt);
\fill [color=zzttqq] (-10.91,-12.13) circle (1.5pt);
\fill [color=zzttqq] (-10.51,-15.63) circle (1.5pt);
\fill [color=zzttqq] (-8.52,-18.54) circle (1.5pt);
\fill [color=qqqqff] (-3.73,-13.39) circle (1.5pt);
\draw[color=black] (-1.81,-9.64) node {$e_1$};
\draw[color=black] (0.61,-8.2) node {$c_1$};
\draw[color=black] (-1.92,-7.57) node {$d_1$};
\draw[color=black] (-1.23,-15.42) node {$e_2$};
\draw[color=black] (-0.72,-17.2) node {$d_2$};
\draw[color=black] (-0.72,-18.7) node {$c_2$};
\draw[color=black] (-6.73,-12.34) node {$e_n$};
\draw[color=black] (-9,-13) node {$d_n$};
\draw[color=black] (-9.48,-12.03) node {$c_n$};
\end{scriptsize}
\end{tikzpicture}
\]
Now we will finish our triangulation by adding a wheel pattern to the center puncture. The arcs added will be labeled as below and there will be $n$ edges added.

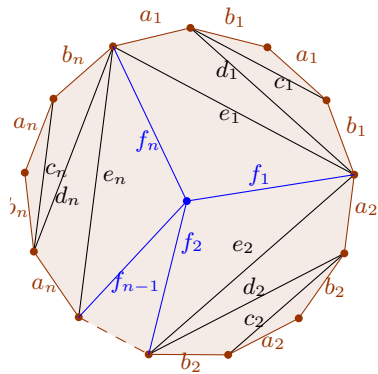
\[\begin{tikzpicture}[line cap=round,line join=round,>=triangle 45,x=0.3cm,y=0.3cm]
\clip(-11.54,-21.4) rectangle (5.26,-4.41);
\fill[color=zzttqq,fill=zzttqq,fill opacity=0.1] (-5.42,-20.19) -- (-1.9,-20.21) -- (1.23,-18.59) -- (3.25,-15.71) -- (3.69,-12.22) -- (2.46,-8.92) -- (-0.16,-6.57) -- (-3.57,-5.71) -- (-7,-6.53) -- (-9.64,-8.85) -- (-10.91,-12.13) -- (-10.51,-15.63) -- (-8.52,-18.54) -- cycle;
\draw [color=zzttqq] (-5.42,-20.19)-- (-1.9,-20.21);
\draw [color=zzttqq] (-1.9,-20.21)-- (1.23,-18.59);
\draw [color=zzttqq] (1.23,-18.59)-- (3.25,-15.71);
\draw [color=zzttqq] (3.25,-15.71)-- (3.69,-12.22);
\draw [color=zzttqq] (3.69,-12.22)-- (2.46,-8.92);
\draw [color=zzttqq] (2.46,-8.92)-- (-0.16,-6.57);
\draw [color=zzttqq] (-0.16,-6.57)-- (-3.57,-5.71);
\draw [color=zzttqq] (-3.57,-5.71)-- (-7,-6.53);
\draw [color=zzttqq] (-7,-6.53)-- (-9.64,-8.85);
\draw [color=zzttqq] (-9.64,-8.85)-- (-10.91,-12.13);
\draw [color=zzttqq] (-10.91,-12.13)-- (-10.51,-15.63);
\draw [color=zzttqq] (-10.51,-15.63)-- (-8.52,-18.54);
\draw [dash pattern=on 4pt off 4pt,color=zzttqq] (-8.52,-18.54)-- (-5.42,-20.19);
\draw (-7,-6.53)-- (3.69,-12.22);
\draw (3.69,-12.22)-- (-3.57,-5.71);
\draw (-3.57,-5.71)-- (2.46,-8.92);
\draw (3.69,-12.22)-- (-5.42,-20.19);
\draw (-5.42,-20.19)-- (3.25,-15.71);
\draw (3.25,-15.71)-- (-1.9,-20.21);
\draw (-8.52,-18.54)-- (-7,-6.53);
\draw (-7,-6.53)-- (-10.51,-15.63);
\draw (-10.51,-15.63)-- (-9.64,-8.85);
\begin{scriptsize}
\fill [color=zzttqq] (-5.42,-20.19) circle (1.5pt);
\fill [color=zzttqq] (-1.9,-20.21) circle (1.5pt);
\draw[color=zzttqq] (-3.53,-20.54) node {$b_2$};
\draw[color=zzttqq] (0.1,-19.68) node {$a_2$};
\draw[color=zzttqq] (2.84,-17.22) node {$b_2$};
\draw[color=zzttqq] (4.24,-13.74) node {$a_2$};
\draw[color=zzttqq] (3.81,-10.07) node {$b_1$};
\draw[color=zzttqq] (1.63,-6.99) node {$a_1$};
\draw[color=zzttqq] (-1.58,-5.23) node {$b_1$};
\draw[color=zzttqq] (-5.29,-5.23) node {$a_1$};
\draw[color=zzttqq] (-8.73,-6.91) node {$b_n$};
\draw[color=zzttqq] (-10.84,-9.96) node {$a_n$};
\draw[color=zzttqq] (-11.23,-13.67) node {$b_n$};
\draw[color=zzttqq] (-10.02,-17.14) node {$a_n$};
\draw[color=zzttqq] (-7.09,-19.64) node {};
\fill [color=zzttqq] (1.23,-18.59) circle (1.5pt);
\fill [color=zzttqq] (3.25,-15.71) circle (1.5pt);
\fill [color=zzttqq] (3.69,-12.22) circle (1.5pt);
\fill [color=zzttqq] (2.46,-8.92) circle (1.5pt);
\fill [color=zzttqq] (-0.16,-6.57) circle (1.5pt);
\fill [color=zzttqq] (-3.57,-5.71) circle (1.5pt);
\fill [color=zzttqq] (-7,-6.53) circle (1.5pt);
\fill [color=zzttqq] (-9.64,-8.85) circle (1.5pt);
\fill [color=zzttqq] (-10.91,-12.13) circle (1.5pt);
\fill [color=zzttqq] (-10.51,-15.63) circle (1.5pt);
\fill [color=zzttqq] (-8.52,-18.54) circle (1.5pt);
\fill [color=qqqqff] (-3.73,-13.39) circle (1.5pt);
\draw[color=black] (-1.81,-9.64) node {$e_1$};
\draw[color=black] (0.61,-8.2) node {$c_1$};
\draw[color=black] (-1.92,-7.57) node {$d_1$};
\draw[color=black] (-1.23,-15.42) node {$e_2$};
\draw[color=black] (-0.72,-17.2) node {$d_2$};
\draw[color=black] (-0.72,-18.7) node {$c_2$};
\draw[color=black] (-6.9,-12.34) node {$e_n$};
\draw[color=black] (-9,-13.15) node {$d_n$};
\draw[color=black] (-9.48,-12.03) node {$c_n$};

\draw[color=blue] (-3.73,-13.39)-- (3.69,-12.22);
\draw[color=blue] (-3.73,-13.39)-- (-5.42,-20.19);
\draw[color=blue] (-3.73,-13.39)-- (-8.52,-18.54);
\draw[color=blue] (-3.73,-13.39)-- (-7,-6.53);
\draw[color=blue] (-.5,-12.25) node {$f_1$};
\draw[color=blue] (-3.5,-15.25) node {$f_2$};
\draw[color=blue] (-6,-17) node {$f_{n-1}$};
\draw[color=blue] (-5.45,-10.75) node {$f_n$};
\end{scriptsize}
\end{tikzpicture}
\]

Now we have completed our desired triangulation $T$ of the surface $(S,M)$. The arcs which are required are $$\{a_1,b_1,c_1,d_1,e_1,f_1,a_2,b_2,c_2,d_2,e_2,f_2, \dots a_n,b_n,c_n,d_n,e_n,f_n\}.$$

Now following the procedure from \cite{fomin} we can construct the quiver $Q_{T_n}$, for the above triangulation.
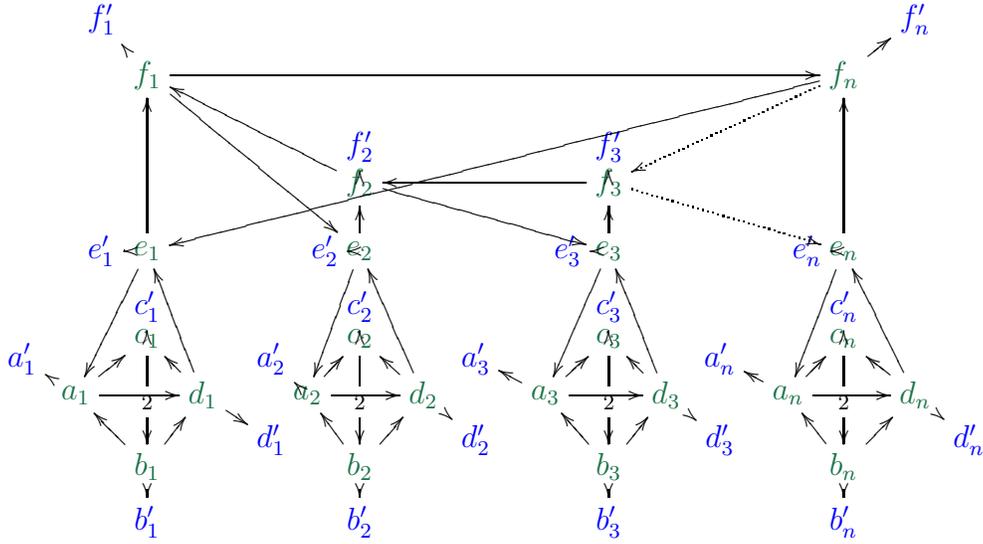
\[\begin{xy} 0;<.6pt,0pt>:<0pt,-.6pt>:: 
(79,146) *+{\color{darkspringgreen}{e_1}} ="0",
(213,146) *+{\color{darkspringgreen}{e_2}} ="1",
(370,146) *+{\color{darkspringgreen}{e_3}} ="2",
(518,146) *+{\color{darkspringgreen}{e_n}} ="3",
(253,237) *+{\color{darkspringgreen}{d_2}} ="4",
(180,237) *+{\color{darkspringgreen}{a_2}} ="5",
(213,201) *+{\color{darkspringgreen}{c_2}} ="6",
(213,282) *+{\color{darkspringgreen}{b_2}}="7",
(370,201) *+{\color{darkspringgreen}{c_3}} ="8",
(330,237) *+{\color{darkspringgreen}{a_3}} ="9",
(406,237) *+{\color{darkspringgreen}{d_3}} ="10",
(370,282) *+{\color{darkspringgreen}{b_3}} ="11",
(518,201) *+{\color{darkspringgreen}{c_n}} ="12",
(563,237) *+{\color{darkspringgreen}{d_n}} ="13",
(483,237) *+{\color{darkspringgreen}{a_n}} ="14",
(518,282) *+{\color{darkspringgreen}{b_n}} ="15",
(79,201) *+{\color{darkspringgreen}{c_1}} ="16",
(114,237) *+{\color{darkspringgreen}{d_1}} ="17",
(79,282) *+{\color{darkspringgreen}{b_1}} ="18",
(34,237) *+{\color{darkspringgreen}{a_1}} ="19",
(79,35) *+{\color{darkspringgreen}{f_1}} ="20",
(213,103) *+{\color{darkspringgreen}{f_2}} ="21",
(370,103) *+{\color{darkspringgreen}{f_3}} ="22",
(518,35) *+{\color{darkspringgreen}{f_n}} ="23",
(50,146) *+{\color{blue}{e_1'}} ="24",
(191,146) *+{\color{blue}{e_2'}} ="25",
(344,146) *+{\color{blue}{e_3'}} ="26",
(495,146) *+{\color{blue}{e_n'}} ="27",
(286,265) *+{\color{blue}{d_2'}} ="28",
(157,214) *+{\color{blue}{a_2'}} ="29",
(213,181) *+{\color{blue}{c_2'}} ="30",
(213,317) *+{\color{blue}{b_2'}} ="31",
(370,181) *+{\color{blue}{c_3'}} ="32",
(286,214) *+{\color{blue}{a_3'}} ="33",
(440,265) *+{\color{blue}{d_3'}} ="34",
(370,317) *+{\color{blue}{b_3'}} ="35",
(518,181) *+{\color{blue}{c_n'}} ="36",
(597,265) *+{\color{blue}{d_n'}} ="37",
(440,214) *+{\color{blue}{a_n'}} ="38",
(518,317) *+{\color{blue}{b_n'}} ="39",
(79,181) *+{\color{blue}{c_1'}} ="40",
(157,265) *+{\color{blue}{d_1'}} ="41",
(79,317) *+{\color{blue}{b_1'}} ="42",
(0,214) *+{\color{blue}{a_1'}} ="43",
(50,0) *+{\color{blue}{f_1'}} ="44",
(213,80) *+{\color{blue}{f_2'}} ="45",
(370,80) *+{\color{blue}{f_3'}} ="46",
(563,0) *+{\color{blue}{f_n'}} ="47",
"17", {\ar"0"},
"0", {\ar"19"},
"0", {\ar"20"},
"23", {\ar"0"},
"0", {\ar"24"},
"4", {\ar"1"},
"1", {\ar"5"},
"20", {\ar"1"},
"1", {\ar"21"},
"1", {\ar"25"},
"2", {\ar"9"},
"10", {\ar"2"},
"21", {\ar"2"},
"2", {\ar"22"},
"2", {\ar"26"},
"13", {\ar"3"},
"3", {\ar"14"},
"22", {\ar@{.>}"3"},
"3", {\ar"23"},
"3", {\ar"27"},
"5", {\ar"4"},
"4", {\ar"6"},
"7", {\ar"4"},
"4", {\ar"28"},
"5", {\ar"6"},
"7", {\ar"5"},
"5", {\ar"29"},
"6", {\ar|*+{\scriptstyle 2}"7"},
"6", {\ar"30"},
"7", {\ar"31"},
"9", {\ar"8"},
"10", {\ar"8"},
"8", {\ar|*+{\scriptstyle 2}"11"},
"8", {\ar"32"},
"9", {\ar"10"},
"11", {\ar"9"},
"9", {\ar"33"},
"11", {\ar"10"},
"10", {\ar"34"},
"11", {\ar"35"},
"13", {\ar"12"},
"14", {\ar"12"},
"12", {\ar|*+{\scriptstyle 2}"15"},
"12", {\ar"36"},
"14", {\ar"13"},
"15", {\ar"13"},
"13", {\ar"37"},
"15", {\ar"14"},
"14", {\ar"38"},
"15", {\ar"39"},
"17", {\ar"16"},
"16", {\ar|*+{\scriptstyle 2}"18"},
"19", {\ar"16"},
"16", {\ar"40"},
"18", {\ar"17"},
"19", {\ar"17"},
"17", {\ar"41"},
"18", {\ar"19"},
"18", {\ar"42"},
"19", {\ar"43"},
"21", {\ar"20"},
"20", {\ar"23"},
"20", {\ar"44"},
"22", {\ar"21"},
"21", {\ar"45"},
"23", {\ar@{.>}"22"},
"22", {\ar"46"},
"23", {\ar"47"},
\end{xy}\]

\section{Statement and Proof of Main Result}
In this paper we will establish a maximal green sequence for the quiver $Q_{T_n}$ constructed above. Our main result is the following:
\begin{theorem}
The quiver $Q_{T_n}$ has a maximal green sequence of 
\[(f_n,f_{n-1},\dots,f_1,f_3,f_4,\dots f_n,\sigma_n,\sigma_{n-1}, \dots \sigma_{1},f_3, f_4, \dots f_n,f_2,f_1,f_n,f_{n-1},\dots f_3, \tau_n,\tau_{n-1},\dots,\tau_1)\]
where $\sigma_i:= (e_i,d_i,b_i,c_i,a_i,b_i,d_i,e_i,c_i,a_i,b_i)$ and $\tau_i:=(e_i,b_i,a_i,c_i,e_i,d_i,b_i,a_i,e_i)$.

\end{theorem}

We will look at the quiver $Q_{T_n}$, and try and break it down into smaller subquivers. The first subquiver of $Q_{T_n}$ we will consider is the oriented $n$-cycle , $C$, which consists of vertices $C_0=\{f_1,f_2,f_3,\dots f_n\}$ and arrows $C_1=\{f_i\rightarrow f_{i-1} | 1\leq i\leq n\ \text{ with } f_0=f_n\}$.

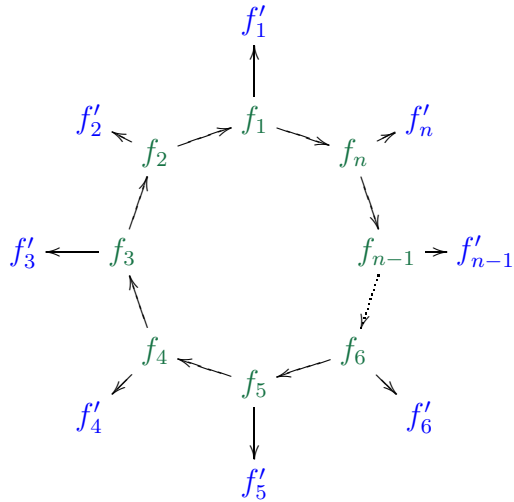
\[\begin{xy}0;<.5pt,0pt>:<0pt,-.5pt>:: 
(75,175) *+{\color{darkspringgreen}{f_3}} ="0",
(100,250) *+{\color{darkspringgreen}{f_4}} ="1",
(175,275) *+{\color{darkspringgreen}{f_5}} ="2",
(275,175) *+{\color{darkspringgreen}{f_{n-1}}} ="3",
(100,100) *+{\color{darkspringgreen}{f_2}} ="4",
(175,75) *+{\color{darkspringgreen}{f_1}} ="5",
(250,100) *+{\color{darkspringgreen}{f_n}} ="6",
(250,250) *+{\color{darkspringgreen}{f_6}} ="7",
(175,0) *+{\color{blue}{f'_1}} ="8",
(300,75) *+{\color{blue}{f'_n}} ="9",
(350,175) *+{\color{blue}{f'_{n-1}}} ="10",
(300,300) *+{\color{blue}{f'_6}} ="11",
(175,350) *+{\color{blue}{f'_5}} ="12",
(50,300) *+{\color{blue}{f'_4}} ="13",
(0,175) *+{\color{blue}{f'_3}} ="14",
(50,75) *+{\color{blue}{f'_2}} ="15",
"1", {\ar"0"},
"0", {\ar"4"},
"0", {\ar"14"},
"2", {\ar"1"},
"1", {\ar"13"},
"7", {\ar"2"},
"2", {\ar"12"},
"6", {\ar"3"},
"3", {\ar@{.>}"7"},
"3", {\ar"10"},
"4", {\ar"5"},
"4", {\ar"15"},
"5", {\ar"6"},
"5", {\ar"8"},
"6", {\ar"9"},
"7", {\ar"11"},
\end{xy}\]

\begin{lemma}[Cycle Lemma]
\label{cycle}
 The sequence $(f_n,f_{n-1}, f_{n-2}, \dots f_1,f_3,f_4,\dots, f_n)$ is a maximal green seqence for the subquiver $C$.
\end{lemma}

\begin{proof}
 First we must check that each mutation which occurs in the sequence occurs at a green vertex. In \cite{brustle} Lemma 2.16 shows that if a vertex $k$ is green in the quiver $Q$, then vertex $k$ is green in the quiver $\mu_j(Q)$ as long as $k \neq j$. Therefore every mutation in the sequence must occur at a green vertex until its second appearance in the sequence. In our case the first $n$ mutations must occur at green vertices.
\vspace{15pt}

In order to understand why the other mutations occur at green vertices it is important to recall from \cite{brustle} that each vertex is either green or red at every mutation step of the sequence. Therefore in order to show that a vertex, $f_k$, is green we must find one arrow $f_k \rightarrow f_j'$ for some $f_j \in C_0$ and all other arrows between $f_k$ and frozen vertices should have $f_k$ as a source as well.

Let us start by considering the quiver, $\hat{C}$, before we have done any mutations. Consider the vertex $f_n$; it is involved in the following arrows: $f_n \rightarrow f_{n-1}$, $f_{1} \rightarrow f_n$, and $f_n \rightarrow f_n'$.  The only arrow with target $f_n$ is the arrow $f_{1} \rightarrow f_n$. In the picture below you see what occurs after our first mutation $\mu_{f_n}$.

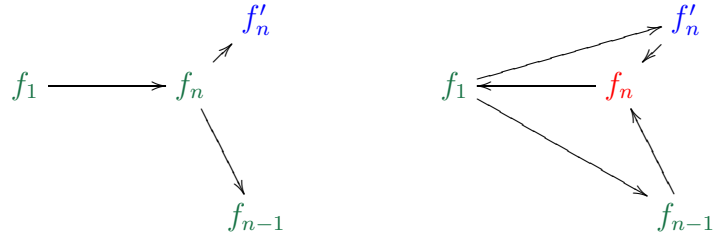
\[\begin{xy} 0;<.5pt,0pt>:<0pt,-.5pt>:: 
(0,50) *+{\color{darkspringgreen}{f_{1}}} ="0",
(125,50) *+{\color{darkspringgreen}{f_n}} ="1",
(175,150) *+{\color{darkspringgreen}{f_{n-1}}} ="2",
(175,0) *+{\color{blue}{f_n'}} ="3",
(325,50) *+{\color{darkspringgreen}{f_{1}}} ="4",
(450,50) *+{\color{red}{f_n}} ="5",
(500,150) *+{\color{darkspringgreen}{f_{n-1}}} ="6",
(500,0) *+{\color{blue}{f_n'}} ="7",
"0", {\ar"1"},
"1", {\ar"2"},
"1", {\ar"3"},
"5", {\ar"4"},
"4", {\ar"6"},
"4", {\ar"7"},
"6", {\ar"5"},
"7", {\ar"5"},
\end{xy}\]
\vspace{15pt} 

At this point the only arrow with target $f_{n-1}$ is the arrow $f_{1} \rightarrow f_{n-1}$. The next mutation is at $f_{n-1}$. After completing this mutation we end up with the diagram below on the left. If we focus only on the subquiver where we delete the vertices $f_n$ and $f_n'$ we obtain the diagram below on the right. It gives us the same diagram that resulted from our mutation at $f_n  $, but we have shifted the  index down one.

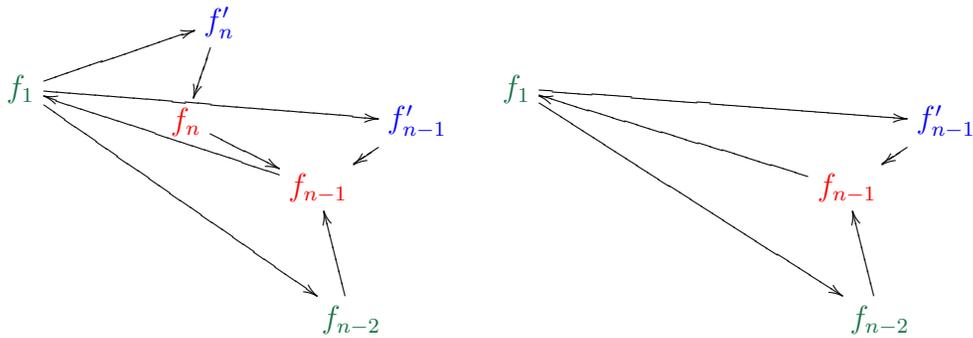
\[\begin{xy} 0;<.5pt,0pt>:<0pt,-.5pt>:: 
(0,50) *+{\color{darkspringgreen}{f_{1}}} ="0",
(125,75) *+{\color{red}{f_n}} ="1",
(225,125) *+{\color{red}{f_{n-1}}} ="2",
(250,225) *+{\color{darkspringgreen}{f_{n-2}}} ="3",
(150,0) *+{\color{blue}{f_n'}}="4",
(300,75) *+{\color{blue}{f_{n-1}'}}="5",
(375,50) *+{\color{darkspringgreen}{f_{1}}} ="6",
(625,125) *+{\color{red}{f_{n-1}}} ="7",
(650,225) *+{\color{darkspringgreen}{f_{n-2}}} ="8",
(700,75) *+{\color{blue}{f_{n-1}'}} ="9",
"2", {\ar"0"},
"0", {\ar"3"},
"0", {\ar"4"},
"0", {\ar"5"},
"1", {\ar"2"},
"4", {\ar"1"},
"3", {\ar"2"},
"5", {\ar"2"},
"7", {\ar"6"},
"6", {\ar"8"},
"6", {\ar"9"},
"8", {\ar"7"},
"9", {\ar"7"},
\end{xy}\]

The most important thing to notice is that each previously mutated vertex remains red, while the only arrows created between the frozen vertices, $\{f_j'\}$, and the mutable vertices, $\{f_j\}$, are the arrows $\{f_1 \rightarrow f_j'\}$. The other important thing to make note of is that the mutation $\mu_{f_{n-1}}$ deletes the arrow $f_1 \leftarrow f_{n}$ which was created by the mutation before it. Since $f_n$ is not adjacent to $f_{n-2}$, and the resulting diagram we get by removing $\{f_n, f_n'\}$ is the same as the previous diagram but with an index shift, we know that this pattern will continue to hold for the mutations $\mu_{n-2},\dots,\mu_{3}$. The resulting quiver $\mu_{f_3}\circ\mu_{f_4}\circ \cdots \circ \mu_{f_n}(Q)$ will be the following:

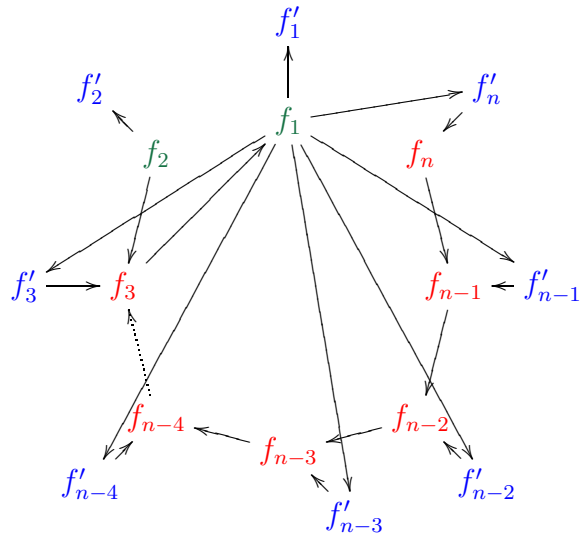
\[\begin{xy} 0;<.5pt,0pt>:<0pt,-.5pt>:: 
(200,75) *+{\color{darkspringgreen}{f_{1}}} ="0",
(300,100) *+{\color{red}{f_{n}}} ="1",
(325,200) *+{\color{red}{f_{n-1}}}="2",
(300,300) *+{\color{red}{f_{n-2}}}="3",
(200,325) *+{\color{red}{f_{n-3}}}="4",
(100,300) *+{\color{red}{f_{n-4}}}="5",
(75,200) *+{\color{red}{f_3}} ="6",
(100,100) *+{\color{darkspringgreen}{f_{2}}} ="7",
(50,50) *+{\color{blue}{f_{2}'}} ="8",
(200,0) *+{\color{blue}{f_{1}'}} ="9",
(350,50) *+{\color{blue}{f_{n}'}}="10",
(400,200) *+{\color{blue}{f_{n-1}'}} ="11",
(350,350) *+{\color{blue}{f_{n-2}'}} ="12",
(250,375) *+{\color{blue}{f_{n-3}'}} ="13",
(50,350) *+{\color{blue}{f_{n-4}'}} ="14",
(0,200) *+{\color{blue}{f_{3}'}} ="15",
"6", {\ar"0"},
"0", {\ar"9"},
"0", {\ar"10"},
"0", {\ar"11"},
"0", {\ar"12"},
"0", {\ar"13"},
"0", {\ar"14"},
"0", {\ar"15"},
"1", {\ar"2"},
"10", {\ar"1"},
"2", {\ar"3"},
"11", {\ar"2"},
"3", {\ar"4"},
"12", {\ar"3"},
"4", {\ar"5"},
"13", {\ar"4"},
"5", {\ar@{.>}"6"},
"14", {\ar"5"},
"7", {\ar"6"},
"15", {\ar"6"},
"7", {\ar"8"},
\end{xy}\]

We can see by looking at the above picture that when we perform the next mutation, no additional arrows will be created from step $(1)$ of the mutation process. Hence the only impact on the quiver will be the reversing of arrows which are incident to the vertex $f_2$.

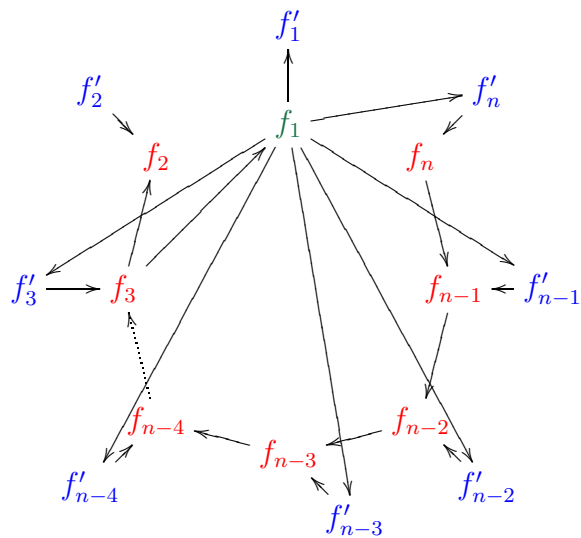
\[\begin{xy} 0;<.5pt,0pt>:<0pt,-.5pt>:: 
(200,75) *+{\color{darkspringgreen}{f_{1}}} ="0",
(300,100) *+{\color{red}{f_{n}}} ="1",
(325,200) *+{\color{red}{f_{n-1}}}="2",
(300,300) *+{\color{red}{f_{n-2}}}="3",
(200,325) *+{\color{red}{f_{n-3}}}="4",
(100,300) *+{\color{red}{f_{n-4}}}="5",
(75,200) *+{\color{red}{f_3}} ="6",
(100,100) *+{\color{red}{f_{2}}} ="7",
(50,50) *+{\color{blue}{f_{2}'}} ="8",
(200,0) *+{\color{blue}{f_{1}'}} ="9",
(350,50) *+{\color{blue}{f_{n}'}}="10",
(400,200) *+{\color{blue}{f_{n-1}'}} ="11",
(350,350) *+{\color{blue}{f_{n-2}'}} ="12",
(250,375) *+{\color{blue}{f_{n-3}'}} ="13",
(50,350) *+{\color{blue}{f_{n-4}'}} ="14",
(0,200) *+{\color{blue}{f_{3}'}} ="15",
"6", {\ar"0"},
"0", {\ar"9"},
"0", {\ar"10"},
"0", {\ar"11"},
"0", {\ar"12"},
"0", {\ar"13"},
"0", {\ar"14"},
"0", {\ar"15"},
"1", {\ar"2"},
"10", {\ar"1"},
"2", {\ar"3"},
"11", {\ar"2"},
"3", {\ar"4"},
"12", {\ar"3"},
"4", {\ar"5"},
"13", {\ar"4"},
"5", {\ar@{.>}"6"},
"14", {\ar"5"},
"6", {\ar"7"},
"15", {\ar"6"},
"8", {\ar"7"},
\end{xy}\] 

Now if we consider the current state of the quiver, there is only one vertex which is green, $f_1$. We notice that the only arrow with target $f_1$ is the arrow $f_3 \rightarrow f_1$. Therefore step $(1)$ of the mutation $\mu_{f_1}$ will only create arrows with source $f_3$. Hence the only possible vertex which could shift from red to green is $f_3$, and in fact $f_3$ will become green. The result of the mutation will be creating the arrows $\{f_3\rightarrow f_j' \text{ }| \text{ }j=n,n-1,n-2,\dots,4\text{ and } j=1\}$. It will also delete the arrow $f_3' \rightarrow f_3$.

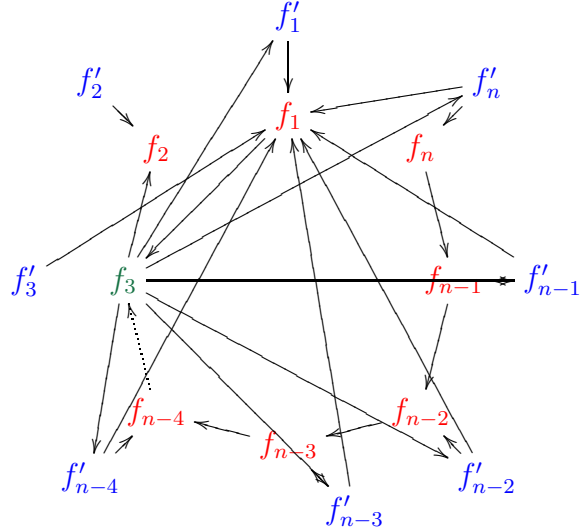
\[
\begin{xy} 0;<.5pt,0pt>:<0pt,-.5pt>:: 
(200,75) *+{\color{red}{f_{1}}} ="0",
(300,100) *+{\color{red}{f_{n}}} ="1",
(325,200) *+{\color{red}{f_{n-1}}}="2",
(300,300) *+{\color{red}{f_{n-2}}}="3",
(200,325) *+{\color{red}{f_{n-3}}}="4",
(100,300) *+{\color{red}{f_{n-4}}}="5",
(75,200) *+{\color{darkspringgreen}{f_3}} ="6",
(100,100) *+{\color{red}{f_{2}}} ="7",
(50,50) *+{\color{blue}{f_{2}'}} ="8",
(200,0) *+{\color{blue}{f_{1}'}} ="9",
(350,50) *+{\color{blue}{f_{n}'}}="10",
(400,200) *+{\color{blue}{f_{n-1}'}} ="11",
(350,350) *+{\color{blue}{f_{n-2}'}} ="12",
(250,375) *+{\color{blue}{f_{n-3}'}} ="13",
(50,350) *+{\color{blue}{f_{n-4}'}} ="14",
(0,200) *+{\color{blue}{f_{3}'}} ="15",
"0", {\ar"6"},
"9", {\ar"0"},
"10", {\ar"0"},
"11", {\ar"0"},
"12", {\ar"0"},
"13", {\ar"0"},
"14", {\ar"0"},
"15", {\ar"0"},
"1", {\ar"2"},
"10", {\ar"1"},
"2", {\ar"3"},
"11", {\ar"2"},
"3", {\ar"4"},
"12", {\ar"3"},
"4", {\ar"5"},
"13", {\ar"4"},
"5", {\ar@{.>}"6"},
"14", {\ar"5"},
"6", {\ar"7"},
"8", {\ar"7"},
"6", {\ar"9"},
"6", {\ar"10"},
"6", {\ar"11"},
"6", {\ar"12"},
"6", {\ar"13"},
"6", {\ar"14"},
\end{xy}\]

We now are forced to mutate at our only green vertex in the quiver. Step $(1)$ of $\mu_{f_3}$  creates the arrows $\{f_1 \rightarrow f_i'\text{ }| \text{ }i=4,5,\dots,n, \text{ and }i=1\}$, but step $(3)$ will delete these arrows since the arrows $\{f_1 \leftarrow f_i'\text{ }| \text{ }i=4,5,\dots,n, \text{ and }i=1\}$ already exist in our current state of the quiver. Meaning the vertex $f_1$ will remain red. The only other arrow whose target is $f_3$ is $f_4 \rightarrow f_3$, so the only vertex which could possibly turn from red to green is $f_4$ and this will occur. Step $(1)$ of the mutation $\mu_{f_3}$ will create the arrows $\{f_4 \rightarrow f_i' \text{ } | \text{ }i=4,5,6,\dots,n \text{ and } i=1\}$, but step $(3)$ will delete the arrow $f_4 \rightarrow f_4'$, because the arrow $f_4' \rightarrow f_4$ is already in the quiver prior to this mutation. 

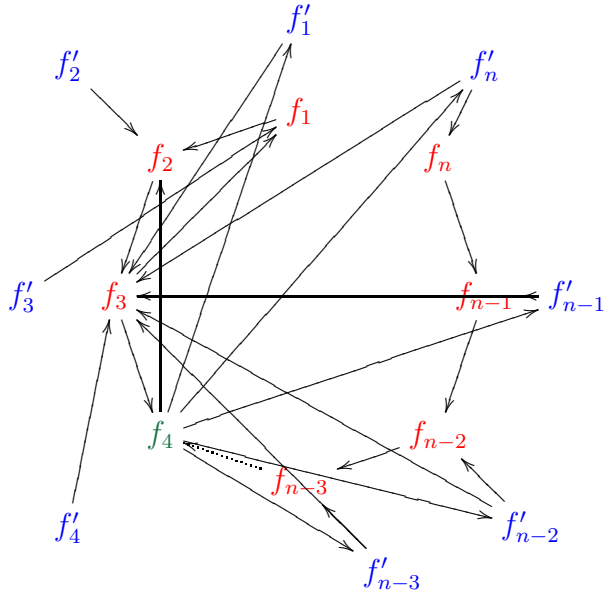
\[
\begin{xy} 0;<.7pt,0pt>:<0pt,-.7pt>:: 
(150,50) *+{\color{red}{f_{1}}} ="0",
(75,75) *+{\color{red}{f_{2}}}="1",
(225,75) *+{\color{red}{f_{n}}} ="2",
(50,150) *+{\color{red}{f_3}} ="3",
(75,225) *+{\color{darkspringgreen}{f_{4}}} ="4",
(150,250) *+{\color{red}{f_{n-3}}}="5",
(225,225) *+{\color{red}{f_{n-2}}}="6",
(250,150) *+{\color{red}{f_{n-1}}}="7",
(150,0) *+{\color{blue}{f_{1}'}} ="8",
(25,25) *+{\color{blue}{f_{2}'}}="9",
(0,150) *+{\color{blue}{f_{3}'}}="10",
(25,275) *+{\color{blue}{f_{4}'}} ="11",
(200,300) *+{\color{blue}{f_{n-3}'}} ="12",
(275,275) *+{\color{blue}{f_{n-2}'}} ="13",
(300,150) *+{\color{blue}{f_{n-1}'}}  ="14",
(250,25) *+{\color{blue}{f_{n}'}} ="15",
"0", {\ar"1"},
"3", {\ar"0"},
"10", {\ar"0"},
"1", {\ar"3"},
"4", {\ar"1"},
"9", {\ar"1"},
"2", {\ar"7"},
"15", {\ar"2"},
"3", {\ar"4"},
"8", {\ar"3"},
"11", {\ar"3"},
"12", {\ar"3"},
"13", {\ar"3"},
"14", {\ar"3"},
"15", {\ar"3"},
"5", {\ar@{.>}"4"},
"4", {\ar"8"},
"4", {\ar"12"},
"4", {\ar"13"},
"4", {\ar"14"},
"4", {\ar"15"},
"6", {\ar"5"},
"12", {\ar"5"},
"7", {\ar"6"},
"13", {\ar"6"},
"14", {\ar"7"},
\end{xy}\]

Our next mutation is then forced to be $\mu_{f_4}$ because it is the only green vertex in the quiver. The only arrows with target $f_4$, are the arrows $f_3 \rightarrow f_4$ and $f_5 \rightarrow f_4$. First let us consider the arrows created with source $f_3$. Step $(1)$ of the mutation process will create the arrows $\{f_3 \rightarrow f_i'\text{ } | \text{ } i=5,6,7,\dots,n \text{ and } i=1\}$. It also creates the arrow $f_3 \rightarrow f_2$. All of these arrows will be deleted by step $(3)$ of the mutation process. This means that no new outgoing arrows are created with source $f_3$, therefore $f_3$ remains a red vertex after mutation. Now we consider the arrows with source $f_5$, which are created by the mutation $\mu_{f_4}$. The arrows created are $\{f_5 \rightarrow f_i'\text{ } | \text{ } i=5,6,7,\dots,n \text{ and } i=1\}$, but the arrow $f_5 \rightarrow f_5'$ is deleted by step $(3)$ of the mutation process.

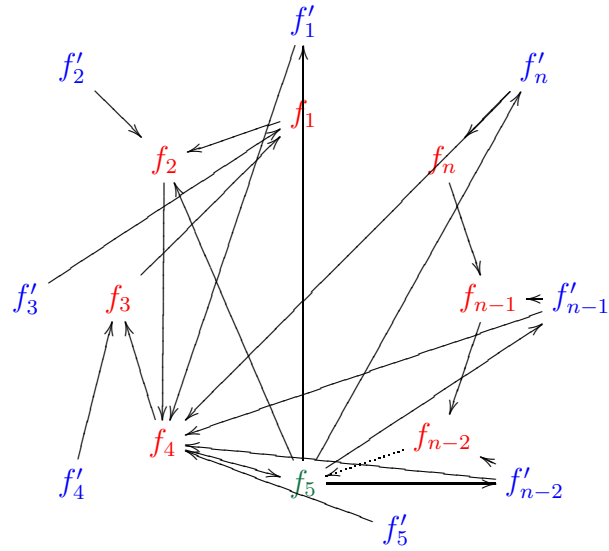
\[\begin{xy} 0;<.7pt,0pt>:<0pt,-.7pt>:: 
(150,50) *+{\color{red}{f_{1}}}="0",
(75,75) *+{\color{red}{f_{2}}}="1",
(50,150) *+{\color{red}{f_{3}}}="2",
(74,228) *+{\color{red}{f_{4}}} ="3",
(150,250) *+{\color{darkspringgreen}{f_{5}}} ="4",
(225,225) *+{\color{red}{f_{n-2}}} ="5",
(250,150) *+{\color{red}{f_{n-1}}} ="6",
(225,75) *+{\color{red}{f_{n}}} ="7",
(200,275) *+{\color{blue}{f_{5}'}} ="8",
(275,250) *+{\color{blue}{f_{n-2}'}} ="9",
(25,250) *+{\color{blue}{f_{4}'}}="10",
(0,150) *+{\color{blue}{f_{3}'}}="11",
(25,25) *+{\color{blue}{f_{2}'}}="12",
(150,0) *+{\color{blue}{f_{1}'}} ="13",
(275,25) *+{\color{blue}{f_{n}'}}="14",
(300,150) *+{\color{blue}{f_{n-1}'}}="15",
"0", {\ar"1"},
"2", {\ar"0"},
"11", {\ar"0"},
"1", {\ar"3"},
"4", {\ar"1"},
"12", {\ar"1"},
"3", {\ar"2"},
"10", {\ar"2"},
"3", {\ar"4"},
"8", {\ar"3"},
"9", {\ar"3"},
"13", {\ar"3"},
"14", {\ar"3"},
"15", {\ar"3"},
"5", {\ar@{.>}"4"},
"4", {\ar"9"},
"4", {\ar"13"},
"4", {\ar"14"},
"4", {\ar"15"},
"6", {\ar"5"},
"9", {\ar"5"},
"7", {\ar"6"},
"15", {\ar"6"},
"14", {\ar"7"},
\end{xy}\]

If we continue this pattern what we are seeing is that by mutating at $f_i$ we are deleting all of the currently existing arrows $\{f_j' \rightarrow f_{i-1}\text{ }|\text{ } j\neq i\}$ and we are creating the arrows $\{f_{i+1} \rightarrow f_j'\text{ }|\text{ } j=i+2,i+3,\dots n \text{ and } j=1\}$. This means at each mutation step, $\mu_{f_i}$, the only vertex which will turn green is $f_{i+1}$. Essentially we are transferring all the outgoing arrows from the vertex $f_i$ to the vertex $f_{i+1}$. This process continues for each mutation in the sequence until the last mutation step. Lets look at the quiver right before this mutation step, $\mu_{f_{n-1}}\circ\mu_{f_{n-2}}\circ \cdots \circ\mu_{f_3}\circ\mu_{f_1}\circ\mu_{f_2}\circ \cdots \circ \mu_{f_n}(Q)$.

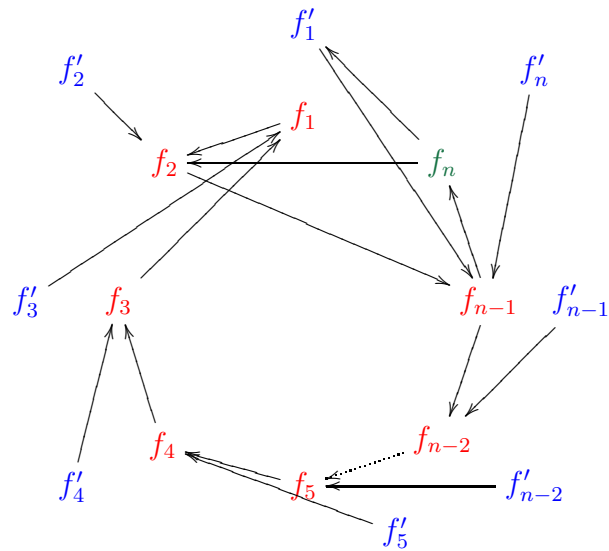
\[\begin{xy} 0;<.7pt,0pt>:<0pt,-.7pt>:: 
(150,50) *+{\color{red}{f_{1}}}="0",
(75,75) *+{\color{red}{f_{2}}}="1",
(50,150) *+{\color{red}{f_{3}}}="2",
(74,228) *+{\color{red}{f_{4}}} ="3",
(150,250) *+{\color{red}{f_{5}}} ="4",
(225,225) *+{\color{red}{f_{n-2}}} ="5",
(250,150) *+{\color{red}{f_{n-1}}} ="6",
(225,75) *+{\color{darkspringgreen}{f_{n}}} ="7",
(200,275) *+{\color{blue}{f_{5}'}} ="8",
(275,250) *+{\color{blue}{f_{n-2}'}} ="9",
(25,250) *+{\color{blue}{f_{4}'}}="10",
(0,150) *+{\color{blue}{f_{3}'}}="11",
(25,25) *+{\color{blue}{f_{2}'}}="12",
(150,0) *+{\color{blue}{f_{1}'}} ="13",
(275,25) *+{\color{blue}{f_{n}'}}="14",
(300,150) *+{\color{blue}{f_{n-1}'}}="15",
"0", {\ar"1"},
"2", {\ar"0"},
"11", {\ar"0"},
"1", {\ar"6"},
"7", {\ar"1"},
"12", {\ar"1"},
"3", {\ar"2"},
"10", {\ar"2"},
"4", {\ar"3"},
"8", {\ar"3"},
"5", {\ar@{.>}"4"},
"9", {\ar"4"},
"6", {\ar"5"},
"15", {\ar"5"},
"6", {\ar"7"},
"13", {\ar"6"},
"14", {\ar"6"},
"7", {\ar"13"},
\end{xy}\]

At this point step $(1)$ of the final mutation in the sequence, will create only the arrows $f_{n-1} \rightarrow f_1'$ and $f_{n-1} \rightarrow f_2$, both of which will be deleted by step $(3)$ of the mutation. Therefore no vertex which is red can become green, meaning that all of the vertices are red. Hence, the seqence of mutations is a maximal green sequence.

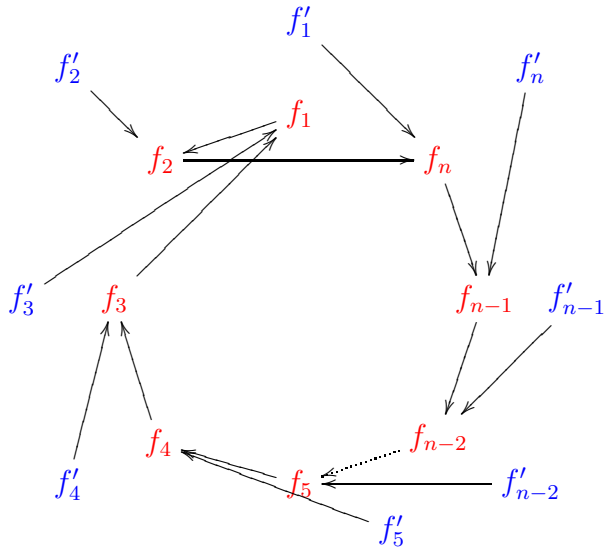
\[\begin{xy} 0;<.7pt,0pt>:<0pt,-.7pt>:: 
(150,50) *+{\color{red}{f_{1}}}="0",
(75,75) *+{\color{red}{f_{2}}}="1",
(50,150) *+{\color{red}{f_{3}}}="2",
(74,228) *+{\color{red}{f_{4}}} ="3",
(150,250) *+{\color{red}{f_{5}}} ="4",
(225,225) *+{\color{red}{f_{n-2}}} ="5",
(250,150) *+{\color{red}{f_{n-1}}} ="6",
(225,75) *+{\color{red}{f_{n}}} ="7",
(200,275) *+{\color{blue}{f_{5}'}} ="8",
(275,250) *+{\color{blue}{f_{n-2}'}} ="9",
(25,250) *+{\color{blue}{f_{4}'}}="10",
(0,150) *+{\color{blue}{f_{3}'}}="11",
(25,25) *+{\color{blue}{f_{2}'}}="12",
(150,0) *+{\color{blue}{f_{1}'}} ="13",
(275,25) *+{\color{blue}{f_{n}'}}="14",
(300,150) *+{\color{blue}{f_{n-1}'}}="15",
"0", {\ar"1"},
"2", {\ar"0"},
"11", {\ar"0"},
"1", {\ar"7"},
"12", {\ar"1"},
"3", {\ar"2"},
"10", {\ar"2"},
"4", {\ar"3"},
"8", {\ar"3"},
"5", {\ar@{.>}"4"},
"9", {\ar"4"},
"6", {\ar"5"},
"15", {\ar"5"},
"7", {\ar"6"},
"14", {\ar"6"},
"13", {\ar"7"},
\end{xy}\]

\end{proof}
The important thing to notice about this sequence is that we pick a starting point and mutate in direction of the cycle until we hit the end of the cycle. At this point we run the mutation sequence backwards from the ending point, but we skip the first two steps of the mutation. We will make use of this sequence again later on in the proof.

Now we must consider what this portion of the sequence does to the rest of the quiver $Q_{T_n}$. Mutation is a local property which only affects adjacent vertices, and since this mutation sequence only involves the vertices $\{f_i\}$ the only vertices that can be affected by the sequence are the vertices $\{f_i\} \cup \{e_i\}$. From the lemma we know that the first part of our sequence, $(f_n,f_{n-1},\dots,f_1,f_3,f_4,\dots f_n),$ is green and that after performing this sequence of mutations all of the vertices $f_i\text { for } 1\leq i \leq n$, will be red. We must now look at what effect the sequence of mutations has on $\{e_i\}$. So we will look at a diagram of the quiver with the vertices $\{e_i \text{ }| \text{ } 1=1,2,\dots,n\}$ drawn in.

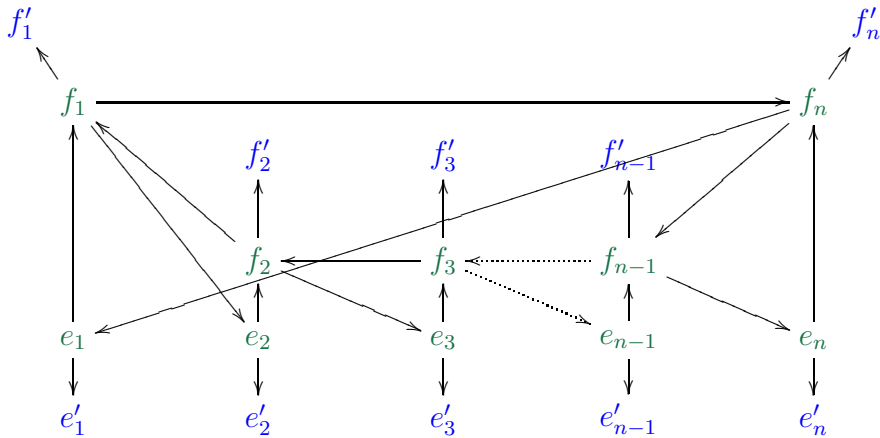
\[\begin{xy} 0;<.4pt,0pt>:<0pt,-.4pt>:: 
(50,300) *+{\color{darkspringgreen}{e_1}} ="0",
(225,300) *+{\color{darkspringgreen}{e_2}} ="1",
(400,300) *+{\color{darkspringgreen}{e_3}} ="2",
(50,75) *+{\color{darkspringgreen}{f_1}} ="3",
(225,225) *+{\color{darkspringgreen}{f_2}} ="4",
(400,225) *+{\color{darkspringgreen}{f_3}} ="5",
(50,375) *+{\color{blue}{e_1'}} ="6",
(225,375) *+{\color{blue}{e_2'}} ="7",
(400,375) *+{\color{blue}{e_3'}} ="8",
(0,0) *+{\color{blue}{f_{1}'}}  ="9",
(225,125) *+{\color{blue}{f_{2}'}}  ="10",
(400,125) *+{\color{blue}{f_{3}'}}  ="11",
(575,300) *+{\color{darkspringgreen}{e_{n-1}}} ="12",
(575,225) *+{\color{darkspringgreen}{f_{n-1}}} ="13",
(750,300) *+{\color{darkspringgreen}{e_n}}="14",
(750,75) *+{\color{darkspringgreen}{f_n}} ="15",
(575,125) *+{\color{blue}{f_{n-1}'}} ="16",
(575,375) *+{\color{blue}{e_{n-1}'}} ="17",
(750,375) *+{\color{blue}{e_n'}}="18",
(800,0) *+{\color{blue}{f_{n}'}}  ="19",
"0", {\ar"3"},
"0", {\ar"6"},
"15", {\ar"0"},
"3", {\ar"1"},
"1", {\ar"4"},
"1", {\ar"7"},
"4", {\ar"2"},
"2", {\ar"5"},
"2", {\ar"8"},
"4", {\ar"3"},
"3", {\ar"9"},
"3", {\ar"15"},
"5", {\ar"4"},
"4", {\ar"10"},
"5", {\ar"11"},
"5", {\ar@{.>}"12"},
"13", {\ar@{.>}"5"},
"12", {\ar"13"},
"12", {\ar"17"},
"13", {\ar"14"},
"15", {\ar"13"},
"13", {\ar"16"},
"14", {\ar"15"},
"14", {\ar"18"},
"15", {\ar"19"},
\end{xy}\]

We see that the initial mutation $\mu_{f_n}$ will result in creating the arrows $e_n \rightarrow f_n'$ and $e_n \rightarrow e_1$. It will delete the arrow $f_{n-1} \rightarrow e_n$. This leaves the vertex $e_n$ not adjacent to any vertex $f_i$, for any $i\neq n$. Meaning that since our sequence consists only of mutations at the vertices $f_i$ until the vertex $f_n$ is mutated at we cannot create new arrows involving $e_n$.

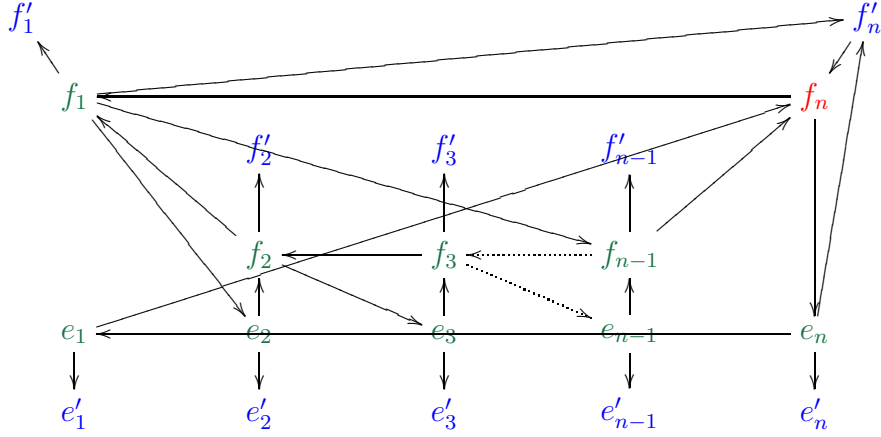
\[\begin{xy} 0;<.4pt,0pt>:<0pt,-.4pt>:: 
(50,300) *+{\color{darkspringgreen}{e_1}} ="0",
(225,300) *+{\color{darkspringgreen}{e_2}} ="1",
(400,300) *+{\color{darkspringgreen}{e_3}} ="2",
(50,75) *+{\color{darkspringgreen}{f_1}} ="3",
(225,225) *+{\color{darkspringgreen}{f_2}} ="4",
(400,225) *+{\color{darkspringgreen}{f_3}} ="5",
(50,375) *+{\color{blue}{e_1'}} ="6",
(225,375) *+{\color{blue}{e_2'}} ="7",
(400,375) *+{\color{blue}{e_3'}} ="8",
(0,0) *+{\color{blue}{f_{1}'}}  ="9",
(225,125) *+{\color{blue}{f_{2}'}}  ="10",
(400,125) *+{\color{blue}{f_{3}'}}  ="11",
(575,300) *+{\color{darkspringgreen}{e_{n-1}}} ="12",
(575,225) *+{\color{darkspringgreen}{f_{n-1}}} ="13",
(750,300) *+{\color{darkspringgreen}{e_n}}="14",
(750,75) *+{\color{red}{f_n}} ="15",
(575,125) *+{\color{blue}{f_{n-1}'}} ="16",
(575,375) *+{\color{blue}{e_{n-1}'}} ="17",
(750,375) *+{\color{blue}{e_n'}}="18",
(800,0) *+{\color{blue}{f_{n}'}}  ="19",
"0", {\ar"6"},
"14", {\ar"0"},
"0", {\ar"15"},
"3", {\ar"1"},
"1", {\ar"4"},
"1", {\ar"7"},
"4", {\ar"2"},
"2", {\ar"5"},
"2", {\ar"8"},
"4", {\ar"3"},
"3", {\ar"9"},
"3", {\ar"13"},
"15", {\ar"3"},
"3", {\ar"19"},
"5", {\ar"4"},
"4", {\ar"10"},
"5", {\ar"11"},
"5", {\ar@{.>}"12"},
"13", {\ar@{.>}"5"},
"12", {\ar"13"},
"12", {\ar"17"},
"13", {\ar"15"},
"13", {\ar"16"},
"15", {\ar"14"},
"14", {\ar"18"},
"14", {\ar"19"},
"19", {\ar"15"},
\end{xy}\]

The next mutation $\mu_{f_{n-1}}$ will create the arrows $e_{n-1} \rightarrow f_{n-1}'$ and $e_{n-1} \rightarrow f_n$. It will also delete the arrow $f_{n-2} \rightarrow e_{n-1}$. In general the mutation step $\mu_{i}$ will create the arrows $e_i \rightarrow f_i'$ and $e_i \rightarrow f_{i+1}$, while deleting the arrow $f_{i-1} \rightarrow e_i$. This pattern holds until we have arrived at the quiver $\mu_{f_3}\circ \mu_{f_4} \circ \cdots \circ \mu_{f_n}(Q_{T_n})$.

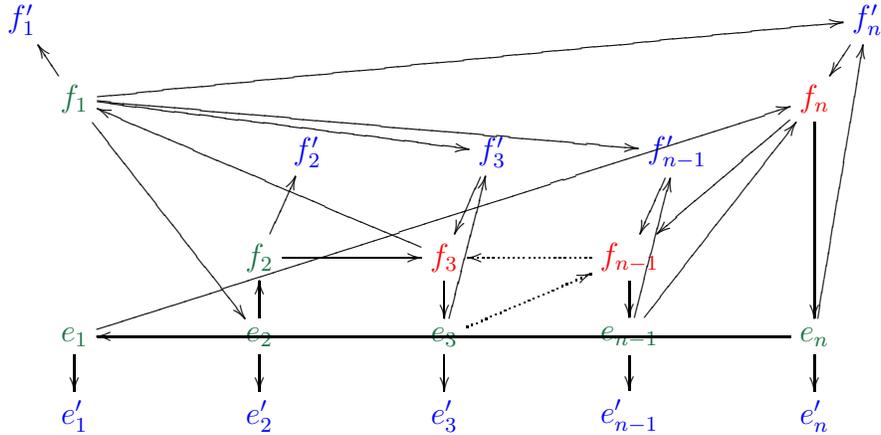
\[\begin{xy} 0;<.4pt,0pt>:<0pt,-.4pt>:: 
(50,300) *+{\color{darkspringgreen}{e_1}} ="0",
(225,300) *+{\color{darkspringgreen}{e_2}} ="1",
(400,300) *+{\color{darkspringgreen}{e_{3}}} ="2",
(50,75) *+{\color{darkspringgreen}{f_1}} ="3",
(225,225) *+{\color{darkspringgreen}{f_2}} ="4",
(400,225) *+{\color{red}{f_{3}}} ="5",
(50,375) *+{\color{blue}{e_1'}} ="6",
(225,375) *+{\color{blue}{e_2'}} ="7",
(400,375) *+{\color{blue}{e_{3}'}} ="8",
(0,0) *+{\color{blue}{f_{1}'}}  ="9",
(270,125) *+{\color{blue}{f_{2}'}}  ="10",
(445,125) *+{\color{blue}{f_{{3}}'}}  ="11",
(575,300) *+{\color{darkspringgreen}{e_{n-1}}} ="12",
(575,225) *+{\color{red}{f_{n-1}}} ="13",
(750,300) *+{\color{darkspringgreen}{e_n}}="14",
(750,75) *+{\color{red}{f_n}} ="15",
(620,125) *+{\color{blue}{f_{n-1}'}} ="16",
(575,375) *+{\color{blue}{e_{n-1}'}} ="17",
(750,375) *+{\color{blue}{e_n'}}="18",
(800,0) *+{\color{blue}{f_{n}'}}  ="19",
"0", {\ar"6"},
"14", {\ar"0"},
"0", {\ar"15"},
"3", {\ar"1"},
"1", {\ar"4"},
"1", {\ar"7"},
"5", {\ar"2"},
"2", {\ar"8"},
"2", {\ar"11"},
"2", {\ar@{.>}"13"},
"5", {\ar"3"},
"3", {\ar"9"},
"3", {\ar"11"},
"3", {\ar"16"},
"3", {\ar"19"},
"4", {\ar"5"},
"4", {\ar"10"},
"11", {\ar"5"},
"13", {\ar@{.>}"5"},
"13", {\ar"12"},
"12", {\ar"15"},
"12", {\ar"16"},
"12", {\ar"17"},
"15", {\ar"13"},
"16", {\ar"13"},
"15", {\ar"14"},
"14", {\ar"18"},
"14", {\ar"19"},
"19", {\ar"15"},
\end{xy}\]
Now we see that at this stage in the mutation sequence we do not have the arrow $f_2 \rightarrow f_1$, and so our next mutation $\mu_{f_2}$ will only create the edges $e_2 \rightarrow f_{3}$ and $e_2 \rightarrow f_2'$. 

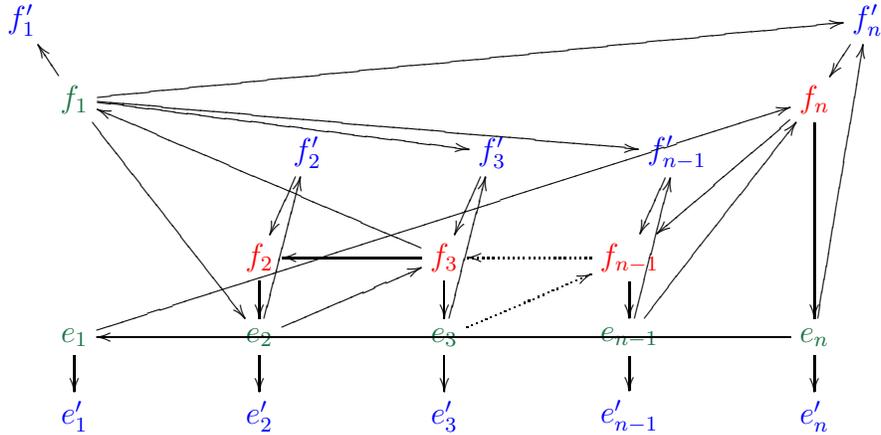
\[\begin{xy} 0;<.4pt,0pt>:<0pt,-.4pt>:: 
(50,300) *+{\color{darkspringgreen}{e_1}} ="0",
(225,300) *+{\color{darkspringgreen}{e_2}} ="1",
(400,300) *+{\color{darkspringgreen}{e_{3}}} ="2",
(50,75) *+{\color{darkspringgreen}{f_1}} ="3",
(225,225) *+{\color{red}{f_2}} ="4",
(400,225) *+{\color{red}{f_{3}}} ="5",
(50,375) *+{\color{blue}{e_1'}} ="6",
(225,375) *+{\color{blue}{e_2'}} ="7",
(400,375) *+{\color{blue}{e_{3}'}} ="8",
(0,0) *+{\color{blue}{f_{1}'}}  ="9",
(270,125) *+{\color{blue}{f_{2}'}}  ="10",
(445,125) *+{\color{blue}{f_{{3}}'}}  ="11",
(575,300) *+{\color{darkspringgreen}{e_{n-1}}} ="12",
(575,225) *+{\color{red}{f_{n-1}}} ="13",
(750,300) *+{\color{darkspringgreen}{e_n}}="14",
(750,75) *+{\color{red}{f_n}} ="15",
(620,125) *+{\color{blue}{f_{n-1}'}} ="16",
(575,375) *+{\color{blue}{e_{n-1}'}} ="17",
(750,375) *+{\color{blue}{e_n'}}="18",
(800,0) *+{\color{blue}{f_{n}'}}  ="19",
"0", {\ar"6"},
"14", {\ar"0"},
"0", {\ar"15"},
"3", {\ar"1"},
"4", {\ar"1"},
"1", {\ar"5"},
"1", {\ar"7"},
"1", {\ar"10"},
"5", {\ar"2"},
"2", {\ar"8"},
"2", {\ar"11"},
"2", {\ar@{.>}"13"},
"5", {\ar"3"},
"3", {\ar"9"},
"3", {\ar"11"},
"3", {\ar"16"},
"3", {\ar"19"},
"5", {\ar"4"},
"10", {\ar"4"},
"11", {\ar"5"},
"13", {\ar@{.>}"5"},
"13", {\ar"12"},
"12", {\ar"15"},
"12", {\ar"16"},
"12", {\ar"17"},
"15", {\ar"13"},
"16", {\ar"13"},
"15", {\ar"14"},
"14", {\ar"18"},
"14", {\ar"19"},
"19", {\ar"15"},
\end{xy}\]

Next, we look at what occurs when we perform the mutation $\mu_{f_1}$. Step $(1)$ of this mutation will create the arrow $e_2 \leftarrow f_3$, but step $(3)$ will delete this arrow because the quiver already has the arrow $e_2 \rightarrow f_3$.

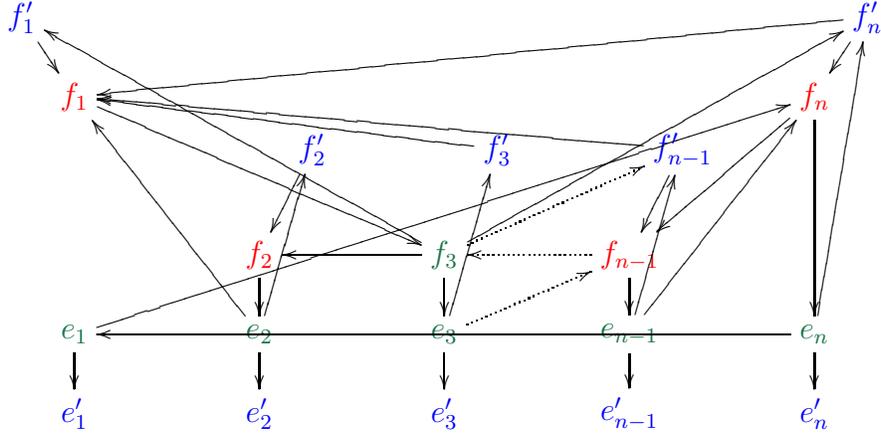
\[\begin{xy} 0;<.4pt,0pt>:<0pt,-.4pt>:: 
(50,300) *+{\color{darkspringgreen}{e_1}} ="0",
(225,300) *+{\color{darkspringgreen}{e_2}} ="1",
(400,300) *+{\color{darkspringgreen}{e_{3}}} ="2",
(50,75) *+{\color{red}{f_1}} ="3",
(225,225) *+{\color{red}{f_2}} ="4",
(400,225) *+{\color{darkspringgreen}{f_{3}}} ="5",
(50,375) *+{\color{blue}{e_1'}} ="6",
(225,375) *+{\color{blue}{e_2'}} ="7",
(400,375) *+{\color{blue}{e_{3}'}} ="8",
(0,0) *+{\color{blue}{f_{1}'}}  ="9",
(275,125) *+{\color{blue}{f_{2}'}}  ="10",
(450,125) *+{\color{blue}{f_{{3}}'}}  ="11",
(575,300) *+{\color{darkspringgreen}{e_{n-1}}} ="12",
(575,225) *+{\color{red}{f_{n-1}}} ="13",
(750,300) *+{\color{darkspringgreen}{e_n}}="14",
(750,75) *+{\color{red}{f_n}} ="15",
(625,125) *+{\color{blue}{f_{n-1}'}} ="16",
(575,375) *+{\color{blue}{e_{n-1}'}} ="17",
(750,375) *+{\color{blue}{e_n'}}="18",
(800,0) *+{\color{blue}{f_{n}'}}  ="19",
"0", {\ar"6"},
"14", {\ar"0"},
"0", {\ar"15"},
"1", {\ar"3"},
"4", {\ar"1"},
"1", {\ar"7"},
"1", {\ar"10"},
"5", {\ar"2"},
"2", {\ar"8"},
"2", {\ar"11"},
"2", {\ar@{.>}"13"},
"3", {\ar"5"},
"9", {\ar"3"},
"11", {\ar"3"},
"16", {\ar"3"},
"19", {\ar"3"},
"5", {\ar"4"},
"10", {\ar"4"},
"5", {\ar"9"},
"13", {\ar@{.>}"5"},
"5", {\ar@{.>}"16"},
"5", {\ar"19"},
"13", {\ar"12"},
"12", {\ar"15"},
"12", {\ar"16"},
"12", {\ar"17"},
"15", {\ar"13"},
"16", {\ar"13"},
"15", {\ar"14"},
"14", {\ar"18"},
"14", {\ar"19"},
"19", {\ar"15"},
\end{xy}\]

Now we notice that at this stage the vertex $e_2$ is not adjacent to any vertices that will be mutated during the remainder of our sequence. Therefore its current arrows will not be affected by the sequence. As we continue performing the mutations of this occurs for each $e_i$ for $i=2,3,\dots n$. More specifically, after the mutation $\mu_{f_i}$ the arrows incident to the vertex $e_i$ will be fixed for the remainder of the mutations in the sequence. This pattern continues until we have the quiver, $\mu_{f_{n-1}}\circ \mu_{f_{n-2}}\circ \cdots \circ \mu_{f_3} \circ \mu_{f_1} \circ \cdots \circ \mu_{f_n}(Q_{T_n})$.

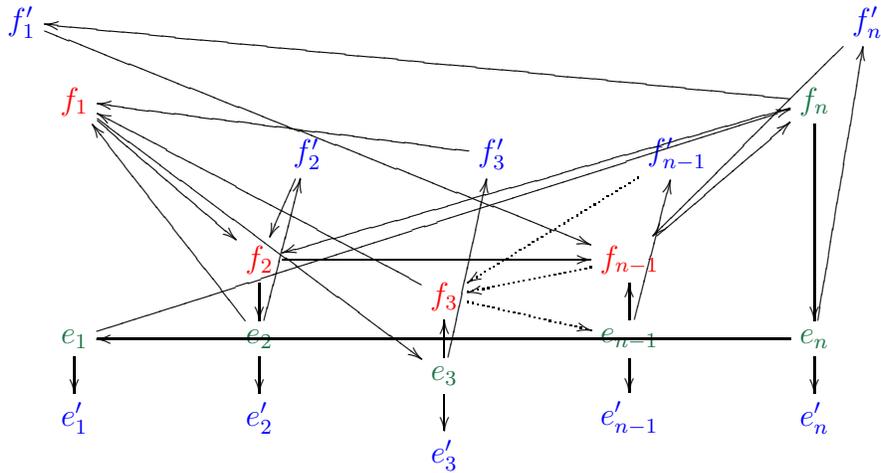
\[\begin{xy} 0;<.4pt,0pt>:<0pt,-.4pt>:: 
(50,300) *+{\color{darkspringgreen}{e_1}} ="0",
(225,300) *+{\color{darkspringgreen}{e_2}} ="1",
(400,335) *+{\color{darkspringgreen}{e_{3}}} ="2",
(50,75) *+{\color{red}{f_1}} ="3",
(225,225) *+{\color{red}{f_2}} ="4",
(400,260) *+{\color{red}{f_{3}}} ="5",
(50,375) *+{\color{blue}{e_1'}} ="6",
(225,375) *+{\color{blue}{e_2'}} ="7",
(400,410) *+{\color{blue}{e_{3}'}} ="8",
(0,0) *+{\color{blue}{f_{1}'}}  ="9",
(270,125) *+{\color{blue}{f_{2}'}}  ="10",
(445,125) *+{\color{blue}{f_{{3}}'}}  ="11",
(575,300) *+{\color{darkspringgreen}{e_{n-1}}} ="12",
(575,225) *+{\color{red}{f_{n-1}}} ="13",
(750,300) *+{\color{darkspringgreen}{e_n}}="14",
(750,75) *+{\color{darkspringgreen}{f_n}} ="15",
(620,125) *+{\color{blue}{f_{n-1}'}} ="16",
(575,375) *+{\color{blue}{e_{n-1}'}} ="17",
(750,375) *+{\color{blue}{e_n'}}="18",
(800,0) *+{\color{blue}{f_{n}'}}  ="19",
"0", {\ar"6"},
"14", {\ar"0"},
"0", {\ar"15"},
"1", {\ar"3"},
"4", {\ar"1"},
"1", {\ar"7"},
"1", {\ar"10"},
"3", {\ar"2"},
"2", {\ar"5"},
"2", {\ar"8"},
"2", {\ar"11"},
"3", {\ar"4"},
"5", {\ar"3"},
"11", {\ar"3"},
"10", {\ar"4"},
"4", {\ar"13"},
"15", {\ar"4"},
"5", {\ar@{.>}"12"},
"13", {\ar@{.>}"5"},
"16", {\ar@{.>}"5"},
"9", {\ar"13"},
"15", {\ar"9"},
"12", {\ar"13"},
"12", {\ar"16"},
"12", {\ar"17"},
"13", {\ar"15"},
"19", {\ar"13"},
"15", {\ar"14"},
"14", {\ar"18"},
"14", {\ar"19"},
\end{xy}\]

Now if we look at the final mutation $\mu_{f_n}$, the net result from the mutation will be creating the arrow $f_{n-1} \rightarrow e_n$ and deleting this arrow $e_n\rightarrow e_1$. It will also create the arrow $e_1\rightarrow f_1'$. The end result is a quiver that up to permuting the vertices $f_1$ and $f_2$, we have the same structure that we had prior to the sequence with the following exceptions: the arrows $\{e_i\rightarrow f_i' \text{ } | \text{ } 1\leq i \leq n\}$ are now in the quiver and the vertices $f_i$ are all red instead of green.

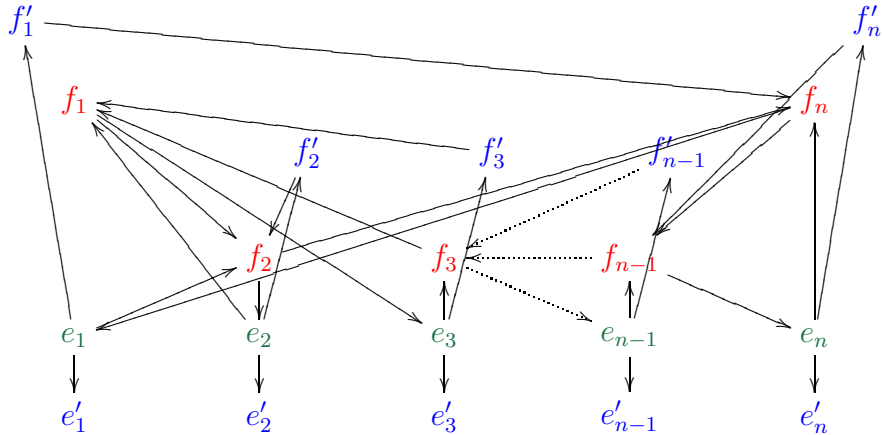
\[\begin{xy} 0;<.4pt,0pt>:<0pt,-.4pt>:: 
(50,300) *+{\color{darkspringgreen}{e_1}} ="0",
(225,300) *+{\color{darkspringgreen}{e_2}} ="1",
(400,300) *+{\color{darkspringgreen}{e_{3}}} ="2",
(50,75) *+{\color{red}{f_1}} ="3",
(225,225) *+{\color{red}{f_2}} ="4",
(400,225) *+{\color{red}{f_{3}}} ="5",
(50,375) *+{\color{blue}{e_1'}} ="6",
(225,375) *+{\color{blue}{e_2'}} ="7",
(400,375) *+{\color{blue}{e_{3}'}} ="8",
(0,0) *+{\color{blue}{f_{1}'}}  ="9",
(270,125) *+{\color{blue}{f_{2}'}}  ="10",
(445,125) *+{\color{blue}{f_{{3}}'}}  ="11",
(575,300) *+{\color{darkspringgreen}{e_{n-1}}} ="12",
(575,225) *+{\color{red}{f_{n-1}}} ="13",
(750,300) *+{\color{darkspringgreen}{e_n}}="14",
(750,75) *+{\color{red}{f_n}} ="15",
(620,125) *+{\color{blue}{f_{n-1}'}} ="16",
(575,375) *+{\color{blue}{e_{n-1}'}} ="17",
(750,375) *+{\color{blue}{e_n'}}="18",
(800,0) *+{\color{blue}{f_{n}'}}  ="19",
"0", {\ar"4"},
"0", {\ar"6"},
"0", {\ar"9"},
"15", {\ar"0"},
"1", {\ar"3"},
"4", {\ar"1"},
"1", {\ar"7"},
"1", {\ar"10"},
"3", {\ar"2"},
"2", {\ar"5"},
"2", {\ar"8"},
"2", {\ar"11"},
"3", {\ar"4"},
"5", {\ar"3"},
"11", {\ar"3"},
"10", {\ar"4"},
"4", {\ar"15"},
"5", {\ar@{.>}"12"},
"13", {\ar@{.>}"5"},
"16", {\ar@{.>}"5"},
"9", {\ar"15"},
"12", {\ar"13"},
"12", {\ar"16"},
"12", {\ar"17"},
"13", {\ar"14"},
"15", {\ar"13"},
"19", {\ar"13"},
"14", {\ar"15"},
"14", {\ar"18"},
"14", {\ar"19"},
\end{xy}\]

This concludes what we need to consider from the initial part of the sequence. Now we must look at the additional pieces which are attached to the bottom of the quiver. We will call these subquivers $H_i$, and define it as $(H_i)_0=\{a_i,b_i,c_i,d_i,e_i,a_i',b_i',c_i',d_i',e_i'\}$ and $(H_i)_1=\{\text{ all arrows between elements of }(H_i)_0\}$. Below is a diagram of $H_i$ after performing the mutation sequence above. We have included in the diagram the vertices which $H_i$ is adjacent to as well, though they are not part of $H_i$.

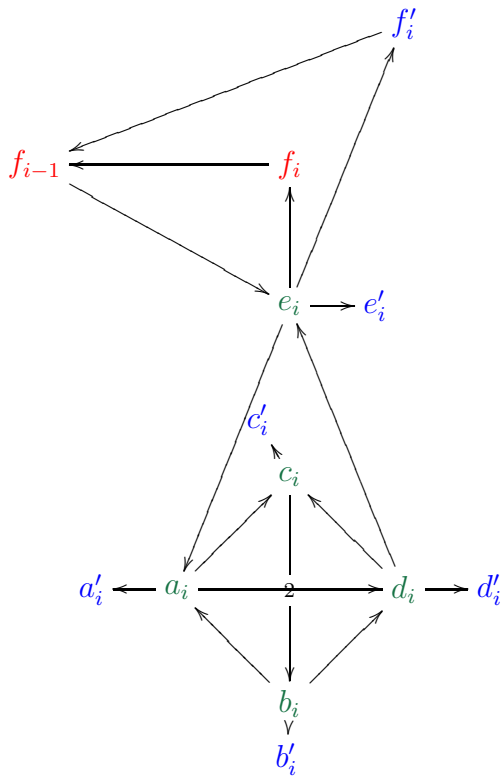
\[\begin{xy} 0;<.6pt,0pt>:<0pt,-.6pt>:: 
(161,178) *+{\color{darkspringgreen}{e_i}} ="0",
(161,286) *+{\color{darkspringgreen}{c_i}} ="1",
(233,357) *+{\color{darkspringgreen}{d_i}} ="2",
(90,357) *+{\color{darkspringgreen}{a_i}} ="3",
(161,429) *+{\color{darkspringgreen}{b_i}} ="4",
(0,89) *+{\color{red}{f_{i-1}}} ="5",
(161,89) *+{\color{red}{f_i} }="6",
(215,178) *+{\color{blue}{e_i'} }="7",
(141,250) *+{\color{blue}{c_i'} }="8",
(287,357) *+{\color{blue}{d_i'} }="9",
(36,357) *+{\color{blue}{a_i'} }="10",
(159,464) *+{\color{blue}{b_i'}} ="11",
(233,0) *+{\color{blue}{f_i'}} ="12",
"2", {\ar"0"},
"0", {\ar"3"},
"5", {\ar"0"},
"0", {\ar"6"},
"0", {\ar"7"},
"0", {\ar"12"},
"2", {\ar"1"},
"3", {\ar"1"},
"1", {\ar|*+{\scriptstyle 2}"4"},
"1", {\ar"8"},
"3", {\ar"2"},
"4", {\ar"2"},
"2", {\ar"9"},
"4", {\ar"3"},
"3", {\ar"10"},
"4", {\ar"11"},
"6", {\ar"5"},
"12", {\ar"5"},
\end{xy}\]

The next part of our maximal green sequence will be mutation sequences that occur only on the vertices of the $H_i$, specifically we will consider what happens when we apply the $\sigma_i$ for each $i$.

We look at what occurs when we perform the mutation sequence $\sigma_i$. Since mutation is a local condition it will only effect the vertices shown in the above diagram, $(H_i)_0 \cup  \{f_i,f_{i-1},f_i'\}$. In addition it is important to note that only arrows between these vertices can be affected by the mutation sequence. Therefore the mutation sequences $\sigma_i$ and $\sigma_j$ will not interact with each other.

By computation we can check the result of performing the sequence of mutations  $\sigma_i$ on the subquiver $H_i$ since it is a finite number of steps. These computations were checked using the java applet developed by Keller. 

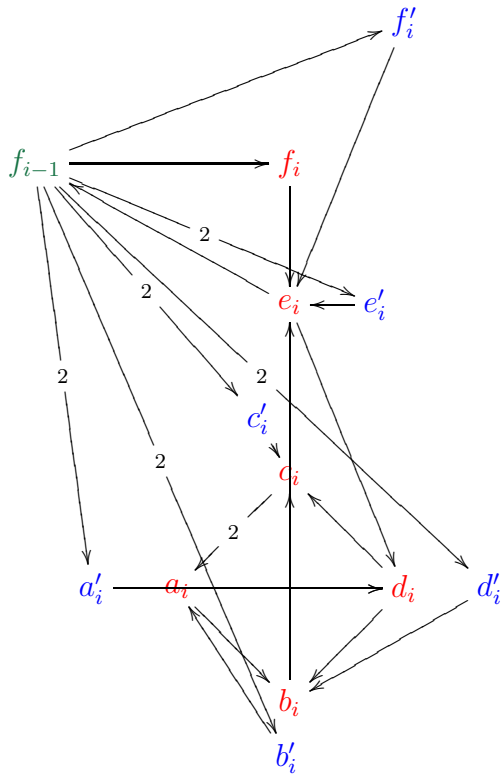
\[\begin{xy} 0;<.6pt,0pt>:<0pt,-.6pt>:: 
(161,178) *+{\color{red}{e_i}} ="0",
(161,286) *+{\color{red}{c_i}} ="1",
(233,357) *+{\color{red}{d_i}} ="2",
(90,357) *+{\color{red}{a_i}} ="3",
(161,429) *+{\color{red}{b_i}} ="4",
(0,89) *+{\color{darkspringgreen}{f_{i-1}}} ="5",
(161,89) *+{\color{red}{f_i} }="6",
(215,178) *+{\color{blue}{e_i'} }="7",
(141,250) *+{\color{blue}{c_i'} }="8",
(287,357) *+{\color{blue}{d_i'} }="9",
(36,357) *+{\color{blue}{a_i'} }="10",
(159,464) *+{\color{blue}{b_i'}} ="11",
(233,0) *+{\color{blue}{f_i'}} ="12",
"0", {\ar"2"},
"4", {\ar"0"},
"0", {\ar"5"},
"6", {\ar"0"},
"7", {\ar"0"},
"12", {\ar"0"},
"2", {\ar"1"},
"1", {\ar|*+{\scriptstyle 2}"3"},
"4", {\ar"1"},
"8", {\ar"1"},
"3", {\ar"2"},
"2", {\ar"4"},
"10", {\ar"2"},
"3", {\ar"4"},
"11", {\ar"3"},
"9", {\ar"4"},
"5", {\ar"6"},
"5", {\ar|*+{\scriptstyle 2}"7"},
"5", {\ar|*+{\scriptstyle 2}"8"},
"5", {\ar|*+{\scriptstyle 2}"9"},
"5", {\ar|*+{\scriptstyle 2}"10"},
"5", {\ar|*+{\scriptstyle 2}"11"},
"5", {\ar"12"},
\end{xy}\]
Notice that each sequence $\sigma_i$ results in the vertex $f_{i-1}$ becoming a green vertex. Therefore the only green vertices in the quiver after performing all of the sequences, $\sigma_i$, are the vertices $\{f_1,f_2,f_3, \dots f_n\}$.

 Another important aspect of the current state of the quiver is that there are no arrows $f_{i-1} \rightarrow a_i$, $f_{i-1} \rightarrow b_i$, $f_{i-1} \rightarrow c_i$, $f_{i-1} \rightarrow d_i$, or arrows in the opposite directions. Therefore when we perform the next portion of the mutation sequence, $(f_n,f_{n-1},\dots f_3,f_1,f_2,\dots f_n)$,  since all of the mutations occur at the vertices $\{f_i\}$ we will not introduce new arrows involving the vertices $\{a_i,b_i,c_i,d_i\}$. 

 Below is a diagram of the current state of the quiver, $\mu_{\sigma_1} \circ \mu_{\sigma_2} \cdots \mu_{\sigma_n} \circ \mu_{f_{n}} \circ \mu_{f_{n-1}} \circ \cdots \circ \mu_{f_3} \circ \mu_{f_1} \circ \cdots \circ \mu_{f_n}(Q_{T_n})$, in which we have omitted all the vertices except for the $\{f_i\}$.
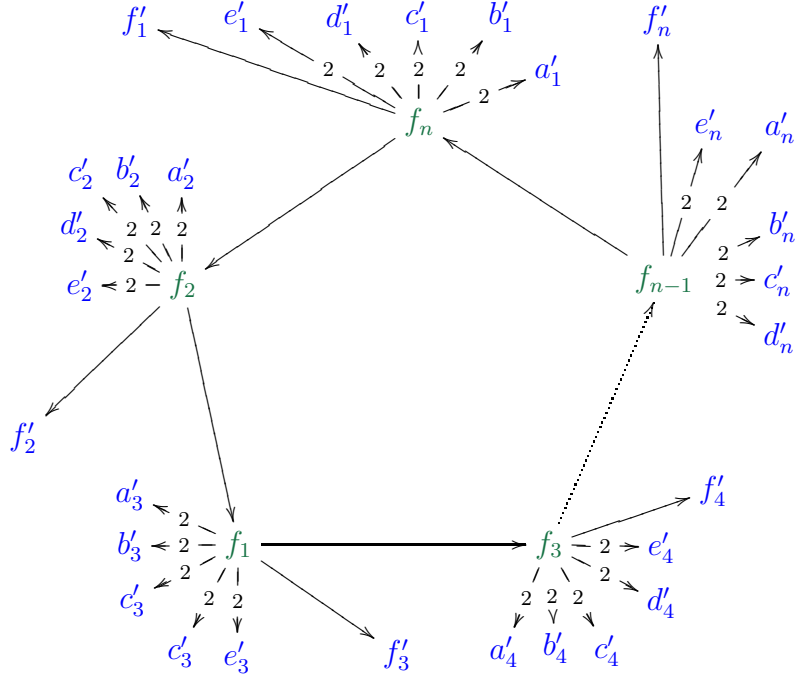
\[\begin{xy} 0;<.5pt,0pt>:<0pt,-.5pt>:: 
(121,203) *+{\color{darkspringgreen}f_2} ="3",
(300,80) *+{\color{darkspringgreen}f_n} ="4",
(163,400) *+{\color{darkspringgreen}f_1} ="5",
(519,82) *+{\color{blue}e_n'} ="6",
(162,0) *+{\color{blue}e_1'} ="7",
(43,204) *+{\color{blue}e_2'} ="8",
(240,2) *+{\color{blue}d_1'} ="9",
(398,41) *+{\color{blue}a_1'} ="10",
(362,0) *+{\color{blue}b_1'} ="11",
(299,0) *+{\color{blue}c_1'} ="12",
(44,119) *+{\color{blue}c_2'} ="13",
(120,119) *+{\color{blue}a_2'} ="14",
(80,117) *+{\color{blue}b_2'} ="15",
(39,159) *+{\color{blue}d_2'} ="16",
(81,401) *+{\color{blue}b_3'} ="17",
(83,442) *+{\color{blue}c_3'} ="18",
(81,362) *+{\color{blue}a_3'} ="19",
(119,481) *+{\color{blue}c_3'} ="20",
(571,200) *+{\color{blue}c_n'} ="21",
(573,84) *+{\color{blue}a_n'} ="22",
(575,159) *+{\color{blue}b_n'} ="23",
(572,243) *+{\color{blue}d_n'} ="24",
(480,3) *+{\color{blue}f_n'} ="25",
(85,4) *+{\color{blue}f_1'} ="26",
(0,319) *+{\color{blue}f_2'} ="27",
(398,400) *+{\color{darkspringgreen}f_{3}} ="29",
(485,199) *+{\color{darkspringgreen}f_{n-1}} ="31",
(283,483) *+{\color{blue}f_{3}'} ="32",
(162,485) *+{\color{blue}e_{3}'} ="33",
(364,481) *+{\color{blue}a_{4}'} ="34",
(442,480) *+{\color{blue}c_{4}'} ="35",
(403,478) *+{\color{blue}b_{4}'} ="36",
(483,444) *+{\color{blue}d_{4}'} ="37",
(483,402) *+{\color{blue}e_{4}'} ="38",
(522,359) *+{\color{blue}f_{4}'} ="39",
"4", {\ar"3"},
"3", {\ar"5"},
"3", {\ar|*+{\scriptstyle 2}"8"},
"3", {\ar|*+{\scriptstyle 2}"13"},
"3", {\ar|*+{\scriptstyle 2}"14"},
"3", {\ar|*+{\scriptstyle 2}"15"},
"3", {\ar|*+{\scriptstyle 2}"16"},
"3", {\ar"27"},
"4", {\ar|*+{\scriptstyle 2}"7"},
"4", {\ar|*+{\scriptstyle 2}"9"},
"4", {\ar|*+{\scriptstyle 2}"10"},
"4", {\ar|*+{\scriptstyle 2}"11"},
"4", {\ar|*+{\scriptstyle 2}"12"},
"4", {\ar"26"},
"31", {\ar"4"},
"5", {\ar|*+{\scriptstyle 2}"17"},
"5", {\ar|*+{\scriptstyle 2}"18"},
"5", {\ar|*+{\scriptstyle 2}"19"},
"5", {\ar|*+{\scriptstyle 2}"20"},
"5", {\ar"29"},
"5", {\ar"32"},
"5", {\ar|*+{\scriptstyle 2}"33"},
"31", {\ar|*+{\scriptstyle 2}"6"},
"31", {\ar|*+{\scriptstyle 2}"21"},
"31", {\ar|*+{\scriptstyle 2}"22"},
"31", {\ar|*+{\scriptstyle 2}"23"},
"31", {\ar|*+{\scriptstyle 2}"24"},
"31", {\ar"25"},
"29", {\ar@{.>}"31"},
"29", {\ar|*+{\scriptstyle 2}"34"},
"29", {\ar|*+{\scriptstyle 2}"35"},
"29", {\ar|*+{\scriptstyle 2}"36"},
"29", {\ar|*+{\scriptstyle 2}"37"},
"29", {\ar|*+{\scriptstyle 2}"38"},
"29", {\ar"39"},
\end{xy}\]

We notice that this diagram is the same as the diagram from Lemma \ref{cycle}, with some minor alterations. First, the cycle has a reversed orientation. We now have attached mulitple frozen vertices to each $f_i$ and we have permuted the vertices $f_1$ and $f_2$. Additionally the indices of the frozen vertices do not match the indices of the mutable vertex they are adjacent to. The important aspect is that we can utilize the same sequence of mutations that we used before to turn all of these vertices red, by adjusting for the new ordering of the vertices $\{f_i\}$. 

We choose a starting point and then mutate in the direction of the cycle, until we reach the end of the cycle, in which case we turn around and run the sequence in reverse, but skipping the first two steps of the sequence. The sequence of mutations we use is the following $(f_3, f_4, \dots f_n,f_2,f_1,f_n,f_{n-1},\dots f_3)$. This sequence is chosen because the result will permute the vertices $f_1$ and $f_2$, undoing the permutation from performing the sequence in Lemm \ref{cycle}. The resulting quiver after performing the sequence of mutations is shown below.  

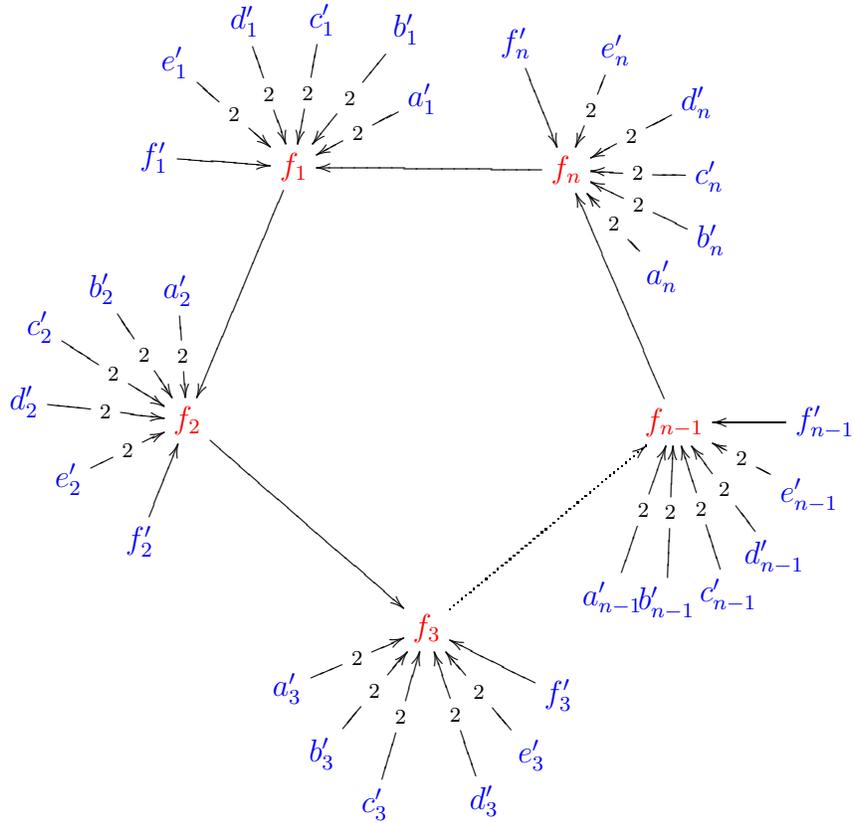
\[\begin{xy} 0;<.6pt,0pt>:<0pt,-.6pt>:: 
(103,253) *+{\color{red}f_2} ="0",
(342,95) *+{\color{red}f_n} ="1",
(170,93) *+{\color{red}f_1} ="2",
(373,19) *+{\color{blue}e_n'} ="3",
(95,27) *+{\color{blue}e_1'} ="4",
(28,290) *+{\color{blue}e_2'} ="5",
(139,1) *+{\color{blue}d_1'} ="6",
(251,50) *+{\color{blue}a_1'} ="7",
(241,6) *+{\color{blue}b_1'} ="8",
(188,0) *+{\color{blue}c_1'} ="9",
(10,194) *+{\color{blue}c_2'} ="10",
(97,171) *+{\color{blue}a_2'} ="11",
(49,170) *+{\color{blue}b_2'} ="12",
(0,241) *+{\color{blue}d_2'} ="13",
(188,463) *+{\color{blue}b_3'} ="14",
(221,495) *+{\color{blue}c_3'} ="15",
(166,422) *+{\color{blue}a_3'} ="16",
(290,493) *+{\color{blue}d_3'} ="17",
(433,139) *+{\color{blue}b_n'} ="18",
(424,52) *+{\color{blue}d_n'} ="19",
(432,99) *+{\color{blue}c_n'} ="20",
(402,162) *+{\color{blue}a_n'} ="21",
(310,15) *+{\color{blue}f_n'} ="22",
(82,87) *+{\color{blue}f_1'} ="23",
(73,329) *+{\color{blue}f_2'} ="24",
(254,384) *+{\color{red}f_3} ="25",
(410,254) *+{\color{red}f_{n-1}} ="26",
(337,426) *+{\color{blue}f_3'} ="27",
(320,464) *+{\color{blue}e_3'} ="28",
(371,365) *+{\color{blue}a_{n-1}'} ="29",
(444,363) *+{\color{blue}c_{n-1}'} ="30",
(405,368) *+{\color{blue}b_{n-1}'} ="31",
(473,339) *+{\color{blue}d_{n-1}'} ="32",
(495,300) *+{\color{blue}e_{n-1}'} ="33",
(505,254) *+{\color{blue}f_{n-1}'} ="34",
"2", {\ar"0"},
"5", {\ar|*+{\scriptstyle 2}"0"},
"10", {\ar|*+{\scriptstyle 2}"0"},
"11", {\ar|*+{\scriptstyle 2}"0"},
"12", {\ar|*+{\scriptstyle 2}"0"},
"13", {\ar|*+{\scriptstyle 2}"0"},
"24", {\ar"0"},
"0", {\ar"25"},
"1", {\ar"2"},
"3", {\ar|*+{\scriptstyle 2}"1"},
"18", {\ar|*+{\scriptstyle 2}"1"},
"19", {\ar|*+{\scriptstyle 2}"1"},
"20", {\ar|*+{\scriptstyle 2}"1"},
"21", {\ar|*+{\scriptstyle 2}"1"},
"22", {\ar"1"},
"26", {\ar"1"},
"4", {\ar|*+{\scriptstyle 2}"2"},
"6", {\ar|*+{\scriptstyle 2}"2"},
"7", {\ar|*+{\scriptstyle 2}"2"},
"8", {\ar|*+{\scriptstyle 2}"2"},
"9", {\ar|*+{\scriptstyle 2}"2"},
"23", {\ar"2"},
"14", {\ar|*+{\scriptstyle 2}"25"},
"15", {\ar|*+{\scriptstyle 2}"25"},
"16", {\ar|*+{\scriptstyle 2}"25"},
"17", {\ar|*+{\scriptstyle 2}"25"},
"25", {\ar@{.>}"26"},
"27", {\ar"25"},
"28", {\ar|*+{\scriptstyle 2}"25"},
"29", {\ar|*+{\scriptstyle 2}"26"},
"30", {\ar|*+{\scriptstyle 2}"26"},
"31", {\ar|*+{\scriptstyle 2}"26"},
"32", {\ar|*+{\scriptstyle 2}"26"},
"33", {\ar|*+{\scriptstyle 2}"26"},
"34", {\ar"26"},
\end{xy}\]

To understand how this sequence will impact the other vertices of the quiver it is important to note that the only vertices which connect to the $\{f_i\}$ at this stage of the quiver are the vertices $\{e_i\}$ and the frozen vertices. Below is a diagram depicting the quiver and the vertices which are adjacent to the $\{f_i\}$.

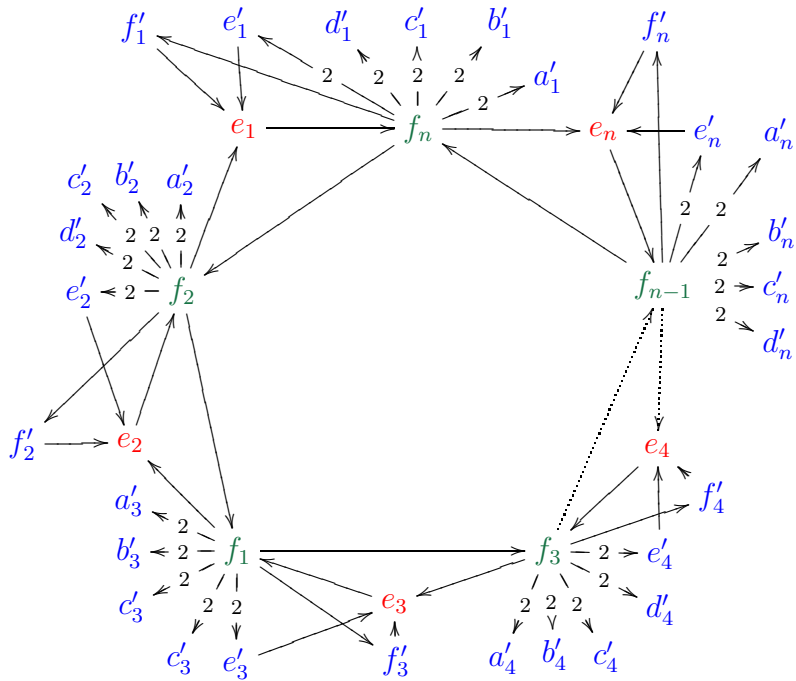
\[\begin{xy} 0;<.5pt,0pt>:<0pt,-.5pt>:: 
(440,82) *+{\color{red}e_n} ="0",
(168,80) *+{\color{red}e_1} ="1",
(82,318) *+{\color{red}e_2} ="2",
(121,203) *+{\color{darkspringgreen}f_2} ="3",
(300,80) *+{\color{darkspringgreen}f_n} ="4",
(163,400) *+{\color{darkspringgreen}f_1} ="5",
(519,82) *+{\color{blue}e_n'} ="6",
(162,0) *+{\color{blue}e_1'} ="7",
(43,204) *+{\color{blue}e_2'} ="8",
(240,2) *+{\color{blue}d_1'} ="9",
(398,41) *+{\color{blue}a_1'} ="10",
(362,0) *+{\color{blue}b_1'} ="11",
(299,0) *+{\color{blue}c_1'} ="12",
(44,119) *+{\color{blue}c_2'} ="13",
(120,119) *+{\color{blue}a_2'} ="14",
(80,117) *+{\color{blue}b_2'} ="15",
(39,159) *+{\color{blue}d_2'} ="16",
(81,401) *+{\color{blue}b_3'} ="17",
(83,442) *+{\color{blue}c_3'} ="18",
(81,362) *+{\color{blue}a_3'} ="19",
(119,481) *+{\color{blue}c_3'} ="20",
(571,200) *+{\color{blue}c_n'} ="21",
(573,84) *+{\color{blue}a_n'} ="22",
(575,159) *+{\color{blue}b_n'} ="23",
(572,243) *+{\color{blue}d_n'} ="24",
(480,3) *+{\color{blue}f_n'} ="25",
(85,4) *+{\color{blue}f_1'} ="26",
(0,319) *+{\color{blue}f_2'} ="27",
(282,440) *+{\color{red}e_{3}} ="28",
(398,400) *+{\color{darkspringgreen}f_{3}} ="29",
(481,322) *+{\color{red}e_4} ="30",
(485,199) *+{\color{darkspringgreen}f_{n-1}} ="31",
(283,483) *+{\color{blue}f_{3}'} ="32",
(162,485) *+{\color{blue}e_{3}'} ="33",
(364,481) *+{\color{blue}a_{4}'} ="34",
(442,480) *+{\color{blue}c_{4}'} ="35",
(403,478) *+{\color{blue}b_{4}'} ="36",
(483,444) *+{\color{blue}d_{4}'} ="37",
(483,402) *+{\color{blue}e_{4}'} ="38",
(522,359) *+{\color{blue}f_{4}'} ="39",
"4", {\ar"0"},
"6", {\ar"0"},
"25", {\ar"0"},
"0", {\ar"31"},
"3", {\ar"1"},
"1", {\ar"4"},
"7", {\ar"1"},
"26", {\ar"1"},
"2", {\ar"3"},
"5", {\ar"2"},
"8", {\ar"2"},
"27", {\ar"2"},
"4", {\ar"3"},
"3", {\ar"5"},
"3", {\ar|*+{\scriptstyle 2}"8"},
"3", {\ar|*+{\scriptstyle 2}"13"},
"3", {\ar|*+{\scriptstyle 2}"14"},
"3", {\ar|*+{\scriptstyle 2}"15"},
"3", {\ar|*+{\scriptstyle 2}"16"},
"3", {\ar"27"},
"4", {\ar|*+{\scriptstyle 2}"7"},
"4", {\ar|*+{\scriptstyle 2}"9"},
"4", {\ar|*+{\scriptstyle 2}"10"},
"4", {\ar|*+{\scriptstyle 2}"11"},
"4", {\ar|*+{\scriptstyle 2}"12"},
"4", {\ar"26"},
"31", {\ar"4"},
"5", {\ar|*+{\scriptstyle 2}"17"},
"5", {\ar|*+{\scriptstyle 2}"18"},
"5", {\ar|*+{\scriptstyle 2}"19"},
"5", {\ar|*+{\scriptstyle 2}"20"},
"28", {\ar"5"},
"5", {\ar"29"},
"5", {\ar"32"},
"5", {\ar|*+{\scriptstyle 2}"33"},
"31", {\ar|*+{\scriptstyle 2}"6"},
"31", {\ar|*+{\scriptstyle 2}"21"},
"31", {\ar|*+{\scriptstyle 2}"22"},
"31", {\ar|*+{\scriptstyle 2}"23"},
"31", {\ar|*+{\scriptstyle 2}"24"},
"31", {\ar"25"},
"29", {\ar"28"},
"32", {\ar"28"},
"33", {\ar"28"},
"30", {\ar"29"},
"29", {\ar@{.>}"31"},
"29", {\ar|*+{\scriptstyle 2}"34"},
"29", {\ar|*+{\scriptstyle 2}"35"},
"29", {\ar|*+{\scriptstyle 2}"36"},
"29", {\ar|*+{\scriptstyle 2}"37"},
"29", {\ar|*+{\scriptstyle 2}"38"},
"29", {\ar"39"},
"31", {\ar@{.>}"30"},
"38", {\ar"30"},
"39", {\ar"30"},
\end{xy}\]

First, we notice that the subquiver including only the vertices $\{f_i\} \cup \{e_i\}$ is exactly the same quiver as the quiver we started with before we did any mutations (with a change of orientation). Therefore since this sequence of mutations is the same as before with an adjustment for this change of orientation we can see that it will have the same effect on the vertices $\{e_i\}$, in terms of creating arrows between the vertices $\{f_i\}$ and $\{e_i\}$. Therefore like before it will not effect the arrows $e_i\rightarrow f_{i-1}$ and $f_n\rightarrow e_n$ except for the fact that the vertices $f_1$ and $f_2$ are permuted by this sequence of muations. Below is a diagram of the end result, with the frozen variables removed to make it easier to see the end result.

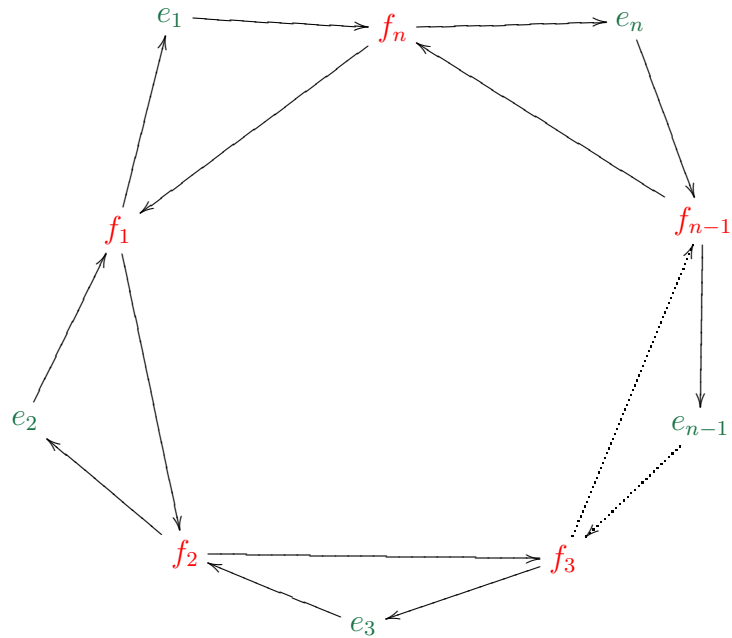
\[\begin{xy} 0;<.5pt,0pt>:<0pt,-.5pt>:: 
(458,3) *+{\color{darkspringgreen}e_n} ="0",
(110,0) *+{\color{darkspringgreen}e_1} ="1",
(0,305) *+{\color{darkspringgreen}e_2} ="2",
(121,406) *+{\color{red}f_2} ="3",
(278,8) *+{\color{red}f_n} ="4",
(70,161) *+{\color{red}f_1} ="5",
(256,461) *+{\color{darkspringgreen}e_3} ="6",
(407,410) *+{\color{red}f_3} ="7",
(511,310) *+{\color{darkspringgreen}e_{n-1}} ="8",
(513,154) *+{\color{red}f_{n-1}} ="9",
"4", {\ar"0"},
"0", {\ar"9"},
"1", {\ar"4"},
"5", {\ar"1"},
"3", {\ar"2"},
"2", {\ar"5"},
"5", {\ar"3"},
"6", {\ar"3"},
"3", {\ar"7"},
"4", {\ar"5"},
"9", {\ar"4"},
"7", {\ar"6"},
"8", {\ar@{.>}"7"},
"7", {\ar@{.>}"9"},
"9", {\ar"8"},
\end{xy}\]

Now the only thing left to do is keep track of the arrows created between the frozen vertices and the vertices $\{e_i\}$ as this sequence of muations is performed. At each initial step of the mutation, $\mu_{f_i}$, the vertex $e_{i+1}$ will gain arrows to each frozen vertex incident to $f_i$. Also important, is that the arrow $e_{i+1}\rightarrow f_{i+1}$ is deleted from this mutation. This means that no additional arrows between $e_{i+1}$ and the frozen vertices will be created during this mutation sequence. Below is a diagram showing this interaction before and after the mutation $\mu_{f_i}$.

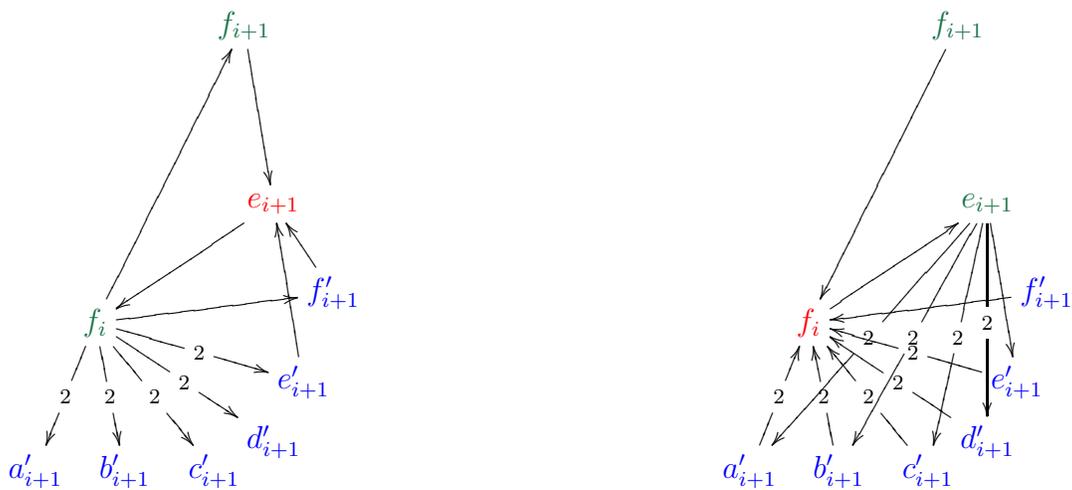
\[\begin{xy} 0;<.45pt,0pt>:<0pt,-.45pt>:: 
(50,250) *+{\color{darkspringgreen}f_i} ="0",
(200,150) *+{\color{red}e_{i+1}} ="1",
(175,0) *+{\color{darkspringgreen}f_{i+1}} ="2",
(0,375) *+{\color{blue}a_{i+1}'} ="3",
(150,375) *+{\color{blue}c_{i+1}'} ="4",
(75,375) *+{\color{blue}b_{i+1}'} ="5",
(200,350) *+{\color{blue}d_{i+1}'} ="6",
(225,300) *+{\color{blue}e_{i+1}'} ="7",
(250,225) *+{\color{blue}f_{i+1}'} ="8",
(650,250) *+{\color{red}f_i} ="9",
(800,150) *+{\color{darkspringgreen}e_{i+1}} ="10",
(775,0) *+{\color{darkspringgreen}f_{i+1}} ="11",
(600,375) *+{\color{blue}a_{i+1}'} ="12",
(750,375) *+{\color{blue}c_{i+1}'} ="13",
(675,375) *+{\color{blue}b_{i+1}'} ="14",
(800,350) *+{\color{blue}d_{i+1}'} ="15",
(825,300) *+{\color{blue}e_{i+1}'} ="16",
(850,225) *+{\color{blue}f_{i+1}'} ="17",
"1", {\ar"0"},
"0", {\ar"2"},
"0", {\ar|*+{\scriptstyle 2}"3"},
"0", {\ar|*+{\scriptstyle 2}"4"},
"0", {\ar|*+{\scriptstyle 2}"5"},
"0", {\ar|*+{\scriptstyle 2}"6"},
"0", {\ar|*+{\scriptstyle 2}"7"},
"0", {\ar"8"},
"2", {\ar"1"},
"7", {\ar"1"},
"8", {\ar"1"},
"9", {\ar"10"},
"11", {\ar"9"},
"12", {\ar|*+{\scriptstyle 2}"9"},
"13", {\ar|*+{\scriptstyle 2}"9"},
"14", {\ar|*+{\scriptstyle 2}"9"},
"15", {\ar|*+{\scriptstyle 2}"9"},
"16", {\ar|*+{\scriptstyle 2}"9"},
"17", {\ar"9"},
"10", {\ar|*+{\scriptstyle 2}"12"},
"10", {\ar|*+{\scriptstyle 2}"13"},
"10", {\ar|*+{\scriptstyle 2}"14"},
"10", {\ar|*+{\scriptstyle 2}"15"},
"10", {\ar"16"},
\end{xy}\]

At this point there is only one part of the sequence left consider: $(\tau_n,\tau_{n-1},\dots,\tau_1)$. Before we do this let us look at the state of the current quiver. We have the cycle below and attached to each vertex is the subquiver $\tilde{H_i}$.

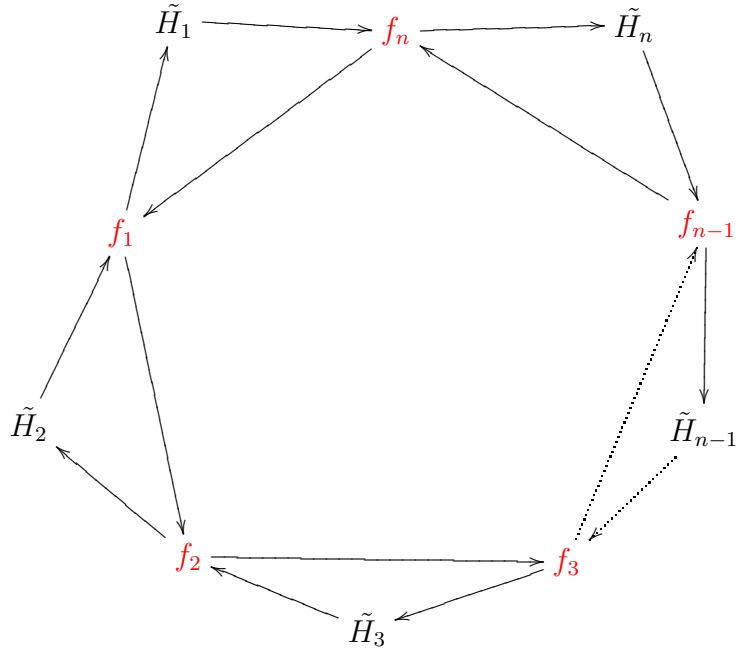
\[\begin{xy} 0;<.5pt,0pt>:<0pt,-.5pt>:: 
(458,3) *+{\tilde{H_n}} ="0",
(110,0) *+{\tilde{H_1}} ="1",
(0,305) *+{\tilde{H_2}} ="2",
(121,406) *+{\color{red}f_2} ="3",
(278,8) *+{\color{red}f_n} ="4",
(70,161) *+{\color{red}f_1} ="5",
(256,461) *+{\tilde{H_3}} ="6",
(407,410) *+{\color{red}f_3} ="7",
(511,310) *+{{\tilde H_{n-1}}} ="8",
(513,154) *+{\color{red}f_{n-1}} ="9",
"4", {\ar"0"},
"0", {\ar"9"},
"1", {\ar"4"},
"5", {\ar"1"},
"3", {\ar"2"},
"2", {\ar"5"},
"5", {\ar"3"},
"6", {\ar"3"},
"3", {\ar"7"},
"4", {\ar"5"},
"9", {\ar"4"},
"7", {\ar"6"},
"8", {\ar@{.>}"7"},
"7", {\ar@{.>}"9"},
"9", {\ar"8"},
\end{xy}\]

Below is a diagram of each $\tilde{H_i}$ and how it connects to the large cycle:
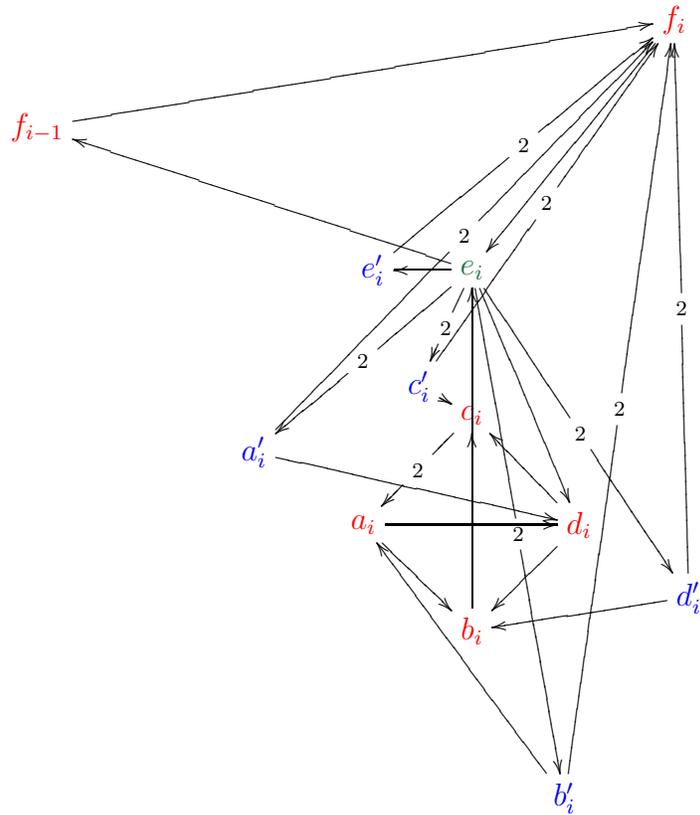
\[\begin{xy} 0;<.5pt,0pt>:<0pt,-.5pt>:: 
(329,189) *+{\color{darkspringgreen}e_i} ="0",
(329,299) *+{\color{red}c_i} ="1",
(410,382) *+{\color{red}d_i} ="2",
(247,382) *+{\color{red}a_i} ="3",
(329,462) *+{\color{red}b_i} ="4",
(0,81) *+{\color{red}f_{i-1}} ="5",
(482,0) *+{\color{red}f_i} ="6",
(254,189) *+{\color{blue}e_i'} ="7",
(289,277) *+{\color{blue}c_i'} ="8",
(493,437) *+{\color{blue}d_i'} ="9",
(164,327) *+{\color{blue}a_i'} ="10",
(399,589) *+{\color{blue}b_i'} ="11",
"0", {\ar"2"},
"4", {\ar"0"},
"0", {\ar"5"},
"6", {\ar"0"},
"0", {\ar"7"},
"0", {\ar|*+{\scriptstyle 2}"8"},
"0", {\ar|*+{\scriptstyle 2}"9"},
"0", {\ar|*+{\scriptstyle 2}"10"},
"0", {\ar|*+{\scriptstyle 2}"11"},
"2", {\ar"1"},
"1", {\ar|*+{\scriptstyle 2}"3"},
"4", {\ar"1"},
"8", {\ar"1"},
"3", {\ar"2"},
"2", {\ar"4"},
"10", {\ar"2"},
"3", {\ar"4"},
"11", {\ar"3"},
"9", {\ar"4"},
"5", {\ar"6"},
"7", {\ar|*+{\scriptstyle 2}"6"},
"8", {\ar|*+{\scriptstyle 2}"6"},
"9", {\ar|*+{\scriptstyle 2}"6"},
"10", {\ar|*+{\scriptstyle 2}"6"},
"11", {\ar|*+{\scriptstyle 2}"6"},
\end{xy}\]

The mutation sequence $\tau_i$ is a sequence only on the of $\tilde{H_i}$ and hence will not effect the vertices of $\tilde{H_j}$ with $j\neq i$. By computation we can check to see that each mutation sequence $\tau_i$ will turn the vertex $e_i$ into a red vertex while leaving all of the other vertices red as well. The end result can be checked by using the Keller mutation applet and is shown below.

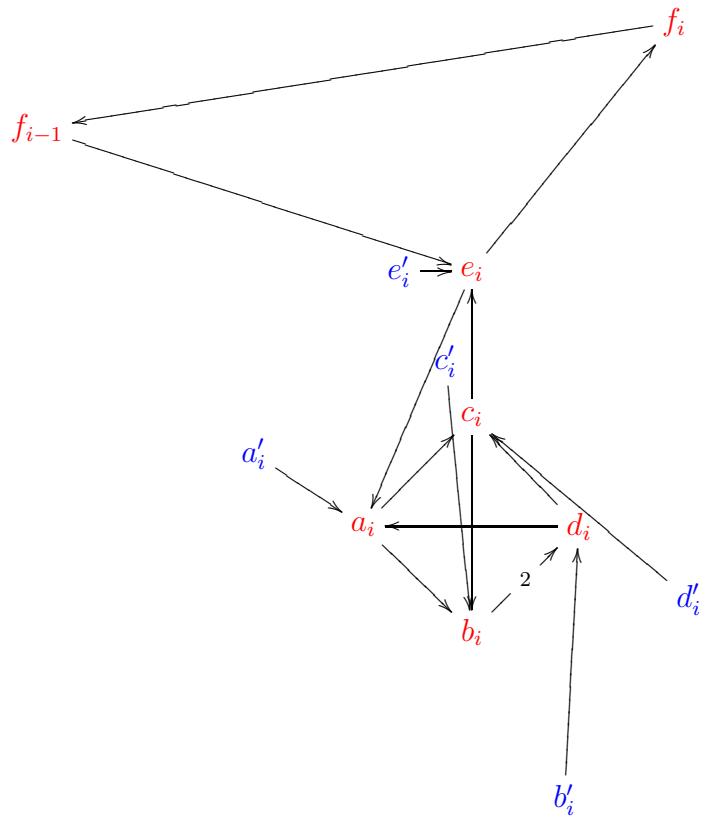
\[\begin{xy} 0;<.5pt,0pt>:<0pt,-.5pt>:: 
(329,189) *+{\color{red}e_i} ="0",
(329,299) *+{\color{red}c_i} ="1",
(410,382) *+{\color{red}d_i} ="2",
(247,382) *+{\color{red}a_i} ="3",
(329,462) *+{\color{red}b_i} ="4",
(0,81) *+{\color{red}f_{i-1}} ="5",
(482,0) *+{\color{red}f_i} ="6",
(274,189) *+{\color{blue}e_i'} ="7",
(309,257) *+{\color{blue}c_i'} ="8",
(493,437) *+{\color{blue}d_i'} ="9",
(164,327) *+{\color{blue}a_i'} ="10",
(399,589) *+{\color{blue}b_i'} ="11",
"1", {\ar"0"},
"0", {\ar"3"},
"5", {\ar"0"},
"0", {\ar"6"},
"7", {\ar"0"},
"2", {\ar"1"},
"3", {\ar"1"},
"1", {\ar"4"},
"9", {\ar"1"},
"2", {\ar"3"},
"4", {\ar|*+{\scriptstyle 2}"2"},
"11", {\ar"2"},
"3", {\ar"4"},
"10", {\ar"3"},
"8", {\ar"4"},
"6", {\ar"5"},
\end{xy}\]

Then the performance of each sequence $\tau_i$ does not create any green vertices. It only turns the vertex $e_i$ from a green vertex to a red vertex. Therefore after completing each mutation sequence $\tau_i$, every vertex in the quiver will be red. This means that the sequence of mutation which we performed was a maximal green sequence. Or in other words that, $$(f_n,f_{n-1},\dots,f_1,f_3,f_4,\dots f_n,\sigma_n,\sigma_{n-1}, \dots \sigma_{1},f_3, f_4, \dots f_n,f_2,f_1,f_n,f_{n-1},\dots f_3, \tau_n,\tau_{n-1},\dots,\tau_1)$$ is a maximal green sequence for the quiver $Q_{T_n}$.

\section*{Acknowledgements}
This paper would not be possible without the helpful insight from M. Yakimov, and the LSU cluster algebra group.


\begin{thebibliography}{99}

\bibitem{alim} M. Alim, S. Cecotti, C. Cordova, S. Espahbodi, A. Rastogi, C. Vafa, \textit{BPS Quivers and Spectra of Complete $N=2$ Quantum Field Theories},Communications in Mathematical Physics. \textbf{323}, (2013) 1185-1127

\bibitem{brustle} Brustle, T., Dupont, G., Perotin, M., \textit{On Maximal Green Sequences}, \textit{International Mathematics Research Notices}. \textbf{16} (2014), 4547-4586

\bibitem{fock} Fock, V., Goncharov, A., \textit{Moduli Spaces of Local Systems and Higher Teichmuller Theory}. Publ. Math. Inst. Hautes Etudes Sci. \textbf{103} (2006), 1-211

\bibitem{fomin} Fomin, S., Shapiro, M., and Thurston, D., \textit{Cluster Algebras and Triangulated Surfaces Part I: Cluster Complexes}, \textit{Acta Math}. \textbf{201} (2008), 83-146

\bibitem{cluster1} Fomin, S., Zelevinsky, A., \textit{Cluster Algebras I: Foundations}, Journal of American Mathematical Society. \textbf{15} (2002), no. 2, 497-529 (electronic). MR 1887642 (2003f: 16050)

\bibitem{garver} Garver, A., Musiker, G., \textit{On Maximal Green Sequences for Type A Quivers} arXiv:1402.6149

\bibitem{gekhtman} Gekhtman, M., Shapiro, M., Vainshtein, A., \textit{Cluster Algebras and Veil-Peterson Forms}, Duke Mathematical Journal \textbf{127} (2005), 291-311

\bibitem{gross} Gross, M., Hacking, P., Keel, S., Kontsevich, M., \textit{Canonical Bases for Cluster Algebras}. arXiv:1411.1394v1


\bibitem{keller} Keller, B., \textit{Quiver Mutation and Combinatorial DT-Invariants}, corrected version of a contribution to DMTCS Proceedings: FPSAC 2013 (2013)

\bibitem{kontsevich} Kontsevich, M., Soibelman, Y., \textit{Stability Structures, Motivic Donaldson-Thomas Invariants and Cluster Transformations} arXiv:0811.2435


\bibitem{williams} Williams, L., \textit{Cluster Algebras: An Introduction}, Bulletin of the Ameican Mathematical Society. \textbf{51} (2014), 1-26

\bibitem{yakimov} Yakimov, M., \textit{Maximal Green Sequences for Double Bruhat Cells.} preprint (2014)



\end{thebibliography}
\end{document}